\documentclass{scrarticle}
\usepackage{graphicx} 
\usepackage[cal=rsfso]{mathalfa} 

\usepackage[utf8]{inputenc}
\usepackage{amsmath,amssymb,amsthm,amsfonts}
\usepackage{mathtools,thmtools}
\usepackage[british]{babel}
\usepackage{csquotes}
\usepackage{hyperref}
\usepackage[nameinlink,noabbrev]{cleveref}
\usepackage{enumerate}
\usepackage{xcolor}
\usepackage{caption}
\usepackage{subcaption}

\usepackage{tikz}
\usepackage{pgfplots}
\pgfplotsset{compat=1.18}

\usepackage[maxbibnames=5,sorting=nyt,sortcites,giveninits=true]{biblatex}
\DeclareNameAlias{sortname}{family-given}
\AtEveryBibitem{
 \clearfield{primaryClass}
}

\newcommand{\R}     {\mathbb{R}}
\newcommand{\C}     {\mathbb{C}}
\newcommand{\N}     {\mathbb{N}}
\newcommand{\Z}     {\mathbb{Z}}
\newcommand{\D}     {\mathcal{D}}
\renewcommand{\P}   {\mathbb{P}}

\newcommand{\diff}  {\mathop{}\!\mathrm{d}}

\newcommand{\rel}   {\mathrm{rel}}

\newcommand{\E}     {\mathcal{E}}
\newcommand{\bbH}   {\mathbb{H}}
\newcommand{\calL}  {\mathcal{L}}
\newcommand{\calA}  {\mathcal{A}}
\newcommand{\calT}  {\mathcal{T}}
\newcommand{\calS}  {\mathcal{S}}
\newcommand{\calH}  {\mathcal{H}}
\newcommand{\calC}  {\mathcal{C}}

\let\Re\relax

\let\phi\varphi
\DeclareMathOperator{\Re}   {Re}

\DeclareMathOperator{\gap}  {gap}
\DeclareMathOperator{\spec} {spec}

\DeclareMathOperator{\Ex}   {\mathbb{E}}

\DeclareMathOperator{\unif} {Unif}
\DeclareMathOperator{\dom}  {Dom}
\DeclareMathOperator{\sing} {s}
\DeclareMathOperator{\divg} {div}

\DeclareMathOperator{\Ran}  {Ran}

\DeclarePairedDelimiterX{\norm}[1]{\lVert}{\rVert}{#1}

\theoremstyle{plain}
\newtheorem{theorem}{Theorem}
\newtheorem{lemma}[theorem]{Lemma}
\newtheorem{corollary}[theorem]{Corollary}

\theoremstyle{definition}
\newtheorem{definition}[theorem]{Definition}
\newtheorem{example}[theorem]{Example}
\newtheorem{remark}[theorem]{Remark}
\newtheorem{remark*}{Remark}

\crefname{lemma}{lemma}{lemmas}
\crefname{theorem}{theorem}{theorems}
\crefname{assumption}{assumption}{assumptions}

\crefname{assumptionalph}{assumption}{Assumptions}

\title{Relaxation times of non-reversible\\ Markov processes}

\author{Andreas Eberle\thanks{E-Mail: \href{mailto:eberle@uni-bonn.de}{eberle@uni-bonn.de}, ORCID: \href{https://orcid.org/0000-0003-0346-3820}{0000-0003-0346-3820}.}\qquad Francis Lörler\thanks{E-Mail: \href{mailto:loerler@uni-bonn.de}{loerler@uni-bonn.de}, ORCID: \href{https://orcid.org/0009-0007-3177-1093}{0009-0007-3177-1093}.}\medskip\\Institute for Applied Mathematics, University of Bonn.}

\begin{document}

\maketitle

\begin{abstract}

We develop a systematic approach to quantify $L^2$-relaxation times for non-reversible Markov processes based on the singular value gap of the generator introduced by Chatterjee. The inverse of the singular value gap is equivalent to the relaxation time of the time-averaged transition semigroup. We show that, moreover, the inverse singular value gap of the two-point motion is equivalent to the relaxation time of the non-averaged transition semigroup, and that it provides a lower bound on the usual spectral gap of the generator.

We then introduce a method for deriving lower bounds on singular value gaps for Markov processes with degenerate noise that is based on the concept of a first- and second-order collapse of the generator. It follows ideas from hypocoercivity developed in a previous series of works  but is simpler and more broadly applicable. In contrast to previous results, it includes settings with non-vanishing first-order collapse, and thus applies directly to Markov chains (in continuous time), but also to diffusion processes and piecewise-deterministic Markov processes.

Our approach yields sharp upper and lower bounds for several classes of examples. First applications include the proof of a conjecture by Diaconis and Miclo on a square-root speed-up for lifted random walks on abelian groups, as well as bounds on relaxation times of switching flows perturbed by noise and of non-reversible diffusion processes.

\begin{samepage}
\par\vspace\baselineskip
\noindent\textbf{Keywords:} Singular value gap; non-reversible; relaxation time; lift; collapse; lifted random walk.\par
\noindent\textbf{MSC Subject Classification:} 60J25, 60J27, 60J60, 47D07.
\end{samepage}
\end{abstract}

\newpage

\section{Introduction}

Suppose that $(p_t)_{t\ge 0}$ is the transition function of a time-homogeneous Markov process 
with state space $S$ and invariant probability measure $\mu$. We assume that $S$ is a Polish space equipped with its Borel $\sigma$-algebra. Let $(P_tf)(x)=\int f(y)\,p_t(x,\diff y)$. By invariance of $\mu$, $(P_t)_{t\ge 0}$ induces a strongly continuous contraction semigroup on the Hilbert space $L^2(\mu)=L^2(S,\mu )$ such that
$$\int_S P_tf\diff\mu \ =\ \int_S f\diff\mu \qquad\text{for all }f\in L^2(\mu )\, .$$
Let $\langle f,g\rangle =\int fg\diff\mu $ denote the $L^2(\mu )$ inner product, and $\|f\| =\langle f,f\rangle^{1/2}$
the corresponding norm. Moreover, let $L^2_0(\mu)=\{ f\in L^2(\mu ):\int f\diff\mu =0\} $ denote the orthogonal complement of the constant functions.
The invariance of $\mu$ implies that $P_t(L^2_0(\mu ))\subseteq L^2_0(\mu )$, and in the following, we will usually consider the restriction of the semigroup to this subspace. 

Approximation of the invariant probability measure in $L^2$ is quantified by
\begin{eqnarray*}
\norm{P_t}_{L_0^2(\mu)\to L_0^2(\mu)} & =& \sup\left\{ \| P_tf\|/\| f\| : f\in L^2_0(\mu )\setminus\{0\}\right\}\\
&=&\sup\left\{ \|P_tf-\int f\diff\mu\| : f\in L^2(\mu )\text{ with }\| f\| =1\right\}   \, .
\end{eqnarray*}
Recall that by duality, this operator norm coincides with the contraction coefficient of the map
$\nu\mapsto \nu p_t$ on probability measures on $S$ w.r.t.\ the chi-square divergence $\chi^2(\cdot\mid\mu )$.
\emph{$L^2$-relaxation} of the Markov process means that $\lim_{t\to\infty}\norm{P_t}_{L_0^2(\mu)\to L_0^2(\mu)}=0 $. Non-asymptotically,
it can quantified by the \emph{$L^2$-relaxation time}
\begin{equation*}
    t_\rel(P)\ = \ \inf\{t\geq 0 : \norm{P_t}_{L_0^2(\mu)\to L_0^2(\mu)}\leq \tfrac{1}{e}\}\,.
\end{equation*}

In the reversible case, where $P_t$ is symmetric on $L^2(\mu )$ for all $t\ge 0$ and the associated generator $(\calL,\dom(\calL))$ is self-adjoint, spectral theory shows that $t_\rel(P)$ coincides with the inverse of the spectral gap 
\begin{equation*}
                \gap(\calL)\ =\ \inf\left\{\Re(\lambda):\lambda\in\spec(-\calL|_{\dom(\calL)\cap L_0^2(\mu)})\right\}
            \end{equation*}
of its generator $\mathcal L$, and thus with the optimal constant in the corresponding Poincar\'e inequality, see e.g.\ \cite[Corollary 8.20]{EberleMP}. This yields a simple variational expression that can be used efficiently to derive both lower and upper bounds for relaxation times of reversible Markov processes \cite{LPW2017markov,MontenegroTetali2006Mathematical}.

In the non-reversible case, however, the above equivalences break down. For non-symmetric semigroups,
it is possible that relaxation does not hold for $(P_t)_{t\ge 0}$ itself, but only for 
the time-averaged semigroup
\begin{equation*}
    \overline P_tf \ = \ \frac{1}{t}\int_0^tP_sf\, \diff s,\qquad f\in L^2(\mu),\ t>0\,.
\end{equation*}
The \emph{$L^2$-relaxation time of time averages} is defined by
\begin{equation*}
    \overline t_\rel(P)\ = \ \inf\bigl\{t\geq 0 : \norm{\overline P_s}_{L_0^2(\mu)\to L_0^2(\mu)}\leq \tfrac 1e \text{ for all }s\geq t\bigr\}\,.
\end{equation*}
Note that, while $t\mapsto\norm{P_t}_{L_0^2(\mu)\to L_0^2(\mu)}$ is non-increasing, this is not necessarily the case for $t\mapsto\norm{\overline P_t}_{L_0^2(\mu)\to L_0^2(\mu)}$. Moreover, while $\norm{P_t}_{L_0^2(\mu)\to L_0^2(\mu)}$ often decays exponentially, $\norm{\overline P_t}_{L_0^2(\mu)\to L_0^2(\mu)}$ usually decays only of order $O(1/t)$.
It is easy to verify that
\begin{equation}
    \label{eq:bartrel_trel} \overline t_\rel(P)\ \le\ \tfrac{e^2}{e-1}\, t_\rel(P)\, ,
\end{equation}
but a converse bound does not hold. For example, for the transition semigroup of deterministic motion on a circle with constant speed, $t_\rel(P)=\infty$ whereas $\overline t_\rel(P)<\infty$.
In contrast to the reversible case, sharp upper and lower bounds for relaxation times are unknown even for some very basic examples of non-reversible Markov processes. 
\begin{example}[Lifted random walks]\label{ex:liftedRW}
    Diaconis et al \cite{Diaconis2000Lift,DiaconisMiclo2013SecondOrder}
    proposed the model of a lifted random walk on a discrete torus $\mathbb Z_n^d$, or, more generally, on an abelian group. It moves persistently with a fixed probability $1-p\in (0,1)$ in a coordinate direction (or a direction given by a generator of the group), and the direction is changed randomly with probability $p$.  
    The process is a Markov process with uniform invariant distribution on the extended state space $S=\mathbb Z_n^d\times V$, where $V$ is the set of possible directions. It has been conjectured in \cite{DiaconisMiclo2013SecondOrder} that for $d$ fixed and $p$ proportional to $1/n$, the relaxation time is ballistic of order $O(n)$, but up to now, this conjecture has only been confirmed in the one-dimensional case by an explicit computation. The corresponding continuous-time Markov chain will be studied in \Cref{sec:Examples}.
\end{example}
The goal of this work is to contribute to the development of efficient and broadly applicable methods to quantify $\overline t_\rel(P)$ and $ t_\rel(P)$ for transition semigroups of non-reversible Markov processes.
In the non-reversible case, the spectral gap of the generator $\mathcal L$ does not have a simple variational expression, and it only controls the asymptotic relaxation rate (under additional assumptions), see the discussion and examples in \Cref{sec:mainresults}.
As a replacement, Chatterjee \cite{Chat2023Spectral} proposed to consider the singular value gap
$$ \sing(\calL)\ =\ \inf\left\{\|\calL f\|/{\norm{f}}:f\in \dom(\calL)\cap L_0^2(\mu), f\neq 0\right\}\, ,$$
see also \cite{Xu2026SVG,HuangLi2026bernstein}. In the reversible case, $\sing(\calL)=\gap(\calL)$, and thus $t_\rel(P)=1/\sing(\calL)$. Moreover, it is relatively easy to show that also in the non-reversible case, the inverse of the singular value gap does provide a sharp control for the relaxation time $\overline t_\rel(P)$, i.e.,
 \begin{equation}\label{eq:bartrel_svg}
    {2(1-e^{-1})}/{\sing(\calL)}\ \leq \ \overline t_\rel(P)\ \leq \ {2e}/{\sing(\calL)} \, ,
\end{equation}
see \Cref{thm:trel_svgap} below. 
In particular, by \eqref{eq:bartrel_trel}, \emph{lower bounds} for relaxation times with and without averaging can be obtained by identifying explicit functions $f\in L^2_0(\mu )$ for which 
$\|\calL f\|/\| f\|$ is small. This is a useful replacement for the common approach in the reversible case of proving lower bounds on relaxation times by identifying functions $f\in L^2_0(\mu )$ for which $\langle f,-\calL f\rangle /\| f\|^2$ is small. Below, the approach will be applied to provide sharp lower bounds in a variety of examples of non-reversible Markov processes.\medskip

These observations raise the question whether the singular value gap can also 
provide \emph{lower bounds} for the spectral gap and \emph{upper bounds} for the non-averaged relaxation time. In general, this cannot be true, as is shown by the counterexample of a deterministic motion with constant speed on a circle. This process has a strictly positive singular value gap but does not converge to stationarity. The results stated in \Cref{ssec:twopointmotion} below show that 
surprisingly, such a control is nevertheless possible if instead of the singular value gap of the generator of the Markov process itself, one considers the singular value gap of the two-point motion consisting of two independent copies of the original process. This is a Markov process with state space $S\times S$, invariant probability measure $\mu\otimes\mu$ and generator
$$\calL^{(2)}\ =\ \calL\otimes I+I\otimes\calL \, ,$$
see \Cref{ssec:twopointmotion} for details. In the reversible case,
\begin{equation*}
    \sing(\calL) = \gap(\calL) = \gap(\calL^{(2)}) = \sing(\calL^{(2)})\, .
\end{equation*}
\Cref{thm:L_L2} below shows that in the general non-reversible case, we still have
    \begin{equation*}
        \gap(\calL)\ \geq\ \frac{1}{2}\sing(\calL^{(2)})\, .
    \end{equation*}
Furthermore, \Cref{cor:twopoint} shows that
\begin{equation}\label{eq:trel_svg2pt}
     (1-e^{-1})/\sing(\calL^{(2)})\ \leq\ t_{\rel}(P)\ \leq\ (\pi+2)/\sing(\calL^{(2)})\,,
\end{equation}
so that $1/\sing(\calL^{(2)})$ is equivalent to the $L^2$-relaxation time of the non-averaged semigroup. Note that in the counterexample of a constant speed motion on the unit circle, the two-point motion is not ergodic and $\sing(\calL^{(2)})=0$, since both copies stay at a fixed distance. \medskip

Having realised that singular value gaps provide a sharp tool to analyse $L^2$ relaxation, 
it is crucial to develop efficient techniques to derive bounds for singular 
value gaps. Upper bounds can often be obtained easily by guessing observables $f\in L^2_0(\mu )$ that relax slowly in the sense that $\| \calL f\|/\| f\|$ is small. Deriving lower bounds for singular value gaps, however, is usually difficult and related to hypocoercivity, which has been a very active research area during the last twenty years \cite{Villani2009Hypocoercivity,DMS2015Hypocoercivity,Eberle2019Langevin,Cao2019Langevin,HerauNier2004Isotropic,Guo2002Landau,Talay2002Stochastic,BLW24Hypocoercivity,bernard2022hypocoercivity,ADNR21,Lu2026EntropyDecay}.

For $L^2$-relaxation of kinetic models two powerful approaches have emerged: the approach 
of Dolbeault, Mouhot and Schmeiser \cite{DMS2015Hypocoercivity} is based on a twisted norm that is adapted to the problem, whereas the elegant approach by space-time Poincar\'e inequalities developed by Armstrong and Mourrat \cite{Albritton2021Variational} and Cao, Lu and Wang \cite{Cao2019Langevin}
is based on time averaging, see also \cite{Lu2022PDMP}. In previous work \cite{EberleLoerler2024Lifts,EberleLoerler2024Divergence}, the authors have connected 
the latter approach to second-order lifts of reversible diffusion processes, see also the reformulation in terms of flow Poincar\'e inequalities in
\cite{EGHLM2025Lifting}. The approach by space-time or flow Poincar\'e inequalities has provided sharp bounds for $L^2$-relaxation times of different second-order lifts of overdamped Langevin dynamics, including in particular (kinetic) Langevin dynamics and randomised Hamiltonian Monte Carlo. However, there are many non-reversible Markov processes that are 
not a pure second-order lift of a reversible Markov process in the sense of \cite{EberleLoerler2024Lifts}. In particular, this is the case for any Markov chain in continuous time.\medskip

In \Cref{ssec:twoscale}, we introduce a new and significantly more general two-component framework
that includes second-order lifts of reversible Markov processes as a special case. We now explain this framework in the case where 
$$S\ =\ X\times V\ =\ \{ (x,v): x\in X, v\in V\}\, $$
is a product space, see \Cref{ssec:twoscale} for the general case. Suppose that
$$\mu (\diff x\diff v)\ =\ \nu (\diff x)\, \kappa (x,\diff v)$$ is the disintegration of $\mu $ into a probability measure $\nu$ on $X$ and a stochastic kernel $\kappa$ from $X$ to $V$.
We define the \emph{symmetric first- and second-order collapse} 
of the generator $\calL$ as non-positive self-adjoint linear operators $\calC_1$ and $\calC_2$ on $L^2_0(X,\nu )$ satisfying
$$\langle f ,\calC_1 f\rangle_{L^2(X,\nu)} =\langle f\circ\pi ,\calL (f\circ\pi )\rangle_{L^2(S,\mu)} \, \quad\text{and}\quad\, \langle f,\calC_2 f\rangle_{L^2(X,\nu)} =-\norm{\calL(f\circ\pi)}^2_{L^2(S,\mu)}  ,$$ 
where $\pi (x,v)=x$, see \Cref{def:collapse} for the precise definition including the domains.
Roughly, $\calC_1$ and $\calC_2$ are related to a first- and second-order expansion of the asymptotics of $P_t(f\circ\pi )$ in the limit $t\downarrow 0$, see \cite{EberleLoerler2024Lifts}.
If $\langle f\circ\pi ,\calL (g\circ\pi )\rangle =0$ for all $f$ and $g$ and if $\frac{1}{2}\calC_2 $ is the generator of a reversible Markov process, then
according to the terminology introduced in \cite{EberleLoerler2024Lifts}, the Markov process with generator $\calL$ is called a \emph{second-order lift} of the process with generator $\frac{1}{2}\calC_2$.

All hypocoercivity approaches described above assume $\calC_1=0$,
and it is not clear how to extend them to non-vanishing first-order collapse.
On the other hand, in many relevant applications, $\calC_1\neq 0$. For example,
for lifted random walks both the first- and the second-order collapse are generators 
of symmetric random walks, see \cite{Loerke2025BA} and \Cref{sec:liftedRW} below.
In other words, a lifted random walk is a classical (first-order) lift of a symmetric random walk (corresponding to the first-order collapse) but it also has a non-trivial second-order collapse. 

Our results in \Cref{ssec:twoscale} below and the applications considered in \Cref{sec:Examples} show that the second-order collapse
is crucial for proving accelerated bounds for relaxation times.
In particular, \Cref{thm:FPI} extends main results in \cite{EberleLoerler2024Lifts, EGHLM2025Lifting} to the case of non-vanishing first-order collapse. The proof is motivated by, but much easier than the proofs in these works. In particular, it completely bypasses the
application of the so-called divergence lemma, a crucial but intricate part of previous hypocoercivity proofs based on space-time Poincar\'e inequalities. While applying our approach directly to the singular value gap of the generator only yields an upper bound on the time-averaged relaxation time, lower bounds on the spectral gap and upper bounds on the non-averaged relaxation time can often be obtained with minimal further effort by considering the two-point motion.

In \Cref{sec:mainresults}, we state the main results on singular value gaps and relaxation times in generality. In \Cref{sec:Examples}, we apply the results to lifted random walks, randomly switching flows, and diffusion processes.
In particular, we prove a conjecture of Diaconis and Miclo on lifted random walks and we 
recover the best known bounds on the spectral gap for (kinetic) Langevin dynamics with a simple proof. The proofs of all results stated in \Cref{sec:mainresults,sec:Examples} are given in \Cref{sec:Proofs1,sec:Proofs2}, respectively.

\section{Singular value gaps and relaxation times}\label{sec:mainresults}

\subsection{Spectral gaps and bounds for relaxation times}\label{ssec:gapbounds}

Let $(\calL,\dom(\calL))$ be the infinitesimal generator of $(P_t)_{t\geq 0}$. Consider the symmetric, negative semi-definite bilinear forms
\begin{align*}
    \E_1^\calL(f,g)\ & =\  -\tfrac{1}{2}\left(\langle f,\calL g\rangle + \langle \calL f,g\rangle\right)\,,&&\kern-5em  f,g\in\dom(\calL),\\
    \E_2^\calL(f,g)\ & =\ \langle \calL f,\calL g\rangle\,,&&\kern-5em f,g\in\dom(\calL)\,.
\end{align*}
The quadratic form $(\E_1^\calL,\dom(\calL))$ is closable while $(\E_2^\calL,\dom(\calL))$ is closed. The domains of the closures are denoted by $\dom (\E_i^\calL)$, $i=1,2$. The generators of these quadratic forms are the 
non-positive self-adjoint linear operators $({\calS}_i,\dom({\calS}_i))$ defined by  
\begin{eqnarray*}
    \langle f,-{\calS}_i g\rangle &=& \E_i^\calL(f,g)\qquad\text{for all }f\in\dom(\E_i^\calL)\text{ and }g\in\dom({\calS}_i)\,  ,\\
    \dom({\calS}_i) &=& \left\{  g\in L^2(\mu ):\exists\, h\in L^2(\mu )\text{ s.t. }\E_i^\calL(f,g)=\langle f,h\rangle\text{ for all }f\in\dom(\E_i^\calL) \right\}\,.
\end{eqnarray*}
Thus ${\calS}_1$ and ${\calS}_2$ are the additive and multiplicative symmetrisations of $\calL$. In particular,
\begin{align*}
   {\calS}_1 g \ &=\ \frac{1}{2}(\calL g+\calL^*g)&&\hspace{-8ex}\text{for } g\in\dom(\calL)\cap\dom(\calL^*),\text{ and}\\
   {\calS}_2 g \ &=\ -\calL^*\calL g &&\hspace{-8ex}\text{for } g\in\dom(\calL^*\calL) . 
\end{align*}
We now introduce the key quantities that can be used to obtain upper and lower bounds for the relaxation times $t_\rel(P)$ and $\overline t_\rel(P)$. We assume that 
$$\D_0 \ \subseteq\ \dom(\calL)\cap L_0^2(\mu)$$
is a core for the restriction of $\calL$ to $\dom(\calL)\cap L_0^2(\mu)$ such that $0\in\D_0$.
\begin{definition}\label{def:gaps}
    \begin{enumerate}[(i)]
        \item The \emph{spectral gap} of $(\calL,\dom(\calL)$ is defined as 
            \begin{equation}\label{defgap}
                \gap(\calL)\ =\ \inf\left\{\Re(\lambda):\lambda\in\spec(-\calL|_{L_0^2(\mu)})\right\}\,.
            \end{equation}
\item The \emph{Poincaré constant} of $(\calL,\dom(\calL))$ is defined as 
            \begin{equation}\label{defPoincare}
                \lambda(\calL)\ = \ \inf\left\{\frac{\E_1^\calL(f,f)}{\norm{f}^2}:f\in \D_0\setminus\{0\}\right\}\,.
            \end{equation}
        \item The \emph{singular value gap} of $(\calL,\dom(\calL))$ is defined as 
            \begin{equation}\label{defsvgap}
                \sing(\calL)\ =\ \inf\left\{\frac{\E_2^\calL(f,f)^{1/2}}{\norm{f}}:f\in \D_0\setminus\{0\}\right\}\,.
            \end{equation}
    \end{enumerate}
\end{definition}
The singular value gap has been introduced by Chatterjee \cite{Chat2023Spectral} in the case of Markov chains. 
Note that since ${\calS}_1$ and ${\calS}_2$ are non-positive self-adjoint operators, 
$$\lambda(\calL)\ =\ \gap({\calS}_1)\qquad\text{and}\qquad\sing(\calL) = \gap({\calS}_2)^{1/2}\, .$$ 
The following relations between these quantities are well-known, see Appendix~\ref{appendix}.
\begin{lemma}\label{lem:gaprelations}
    It holds
    $\gap(\calL)\geq\lambda(\calL)\, $ and $\,\sing(\calL)\geq\lambda(\calL)\,.$
\end{lemma}

In the reversible case, $(\calL,\dom(\calL))$ is self-adjoint in $L^2(\mu)$, the transition semigroup can be written as $P_t = \exp(t\calL)$ by the spectral theorem (see e.g.\ \cite{ReedSimon1980MethodsI}), and thus
\begin{equation*}
    \gap(\calL)\ =\ \lambda(\calL)\ =\ \sing(\calL)\,,\quad\text{ and }\quad t_\rel(P)\ =\ 1/{\gap(\calL)}\, .
\end{equation*}

In the non-reversible case, however, these quantities do not agree in general, and can even be of different orders. It is well-known that under appropriate additional conditions ensuring that a spectral mapping theorem holds, the spectral gap determines the \emph{asymptotic} $L^2$-convergence rate in the limit $t\to\infty$, i.e.,
\begin{equation}\label{asymptoticrate}
    \gap(\calL)\ =\ -\lim_{t\to\infty}\tfrac{1}{t}\log\norm{P_t}_{L_0^2(\mu)\to L_0^2(\mu)}\,,
\end{equation}
see \cite[Chapter IV.3]{engel1999semigroups}.
By a classical Grönwall argument, the Poincar\'e constant yields the non-asymptotic upper bound 
\begin{equation}\label{eq:trelpoincare}
    t_\rel(P)\ \leq\ {1}/{\lambda(\calL)}
\end{equation}
on the relaxation time, see e.g.\ \cite[Theorem 4.2.5]{BGL2014Analysis}, but this can be far off in the non-reversible case, where often the relaxation time is finite even when $\lambda(\calL) = 0$.

\begin{example}[Diffusions on the circle]\label{ex:rotation}
    Suppose that $\mu$ is the uniform distribution on the circle $S^1_\ell=\mathbb R/(\ell\mathbb Z)$, $\ell\in (0,\infty )$, and $\calL f=\sigma^2f''+bf'$ with coefficients $\sigma\in [0,\infty )$ and $ b\in\mathbb R$, and domain $H^{2,2}(S^1_\ell)$ for $\sigma >0$ and $H^{1,2}(S^1_\ell)$ for $ \sigma =0${ and } $b\neq 0$. 
    Then the eigenfunctions of $-\calL$ and $\calL^*\calL$ are the trigonometric functions $\exp(2\pi ikx/\ell)$, $k\in\mathbb Z$, with eigenvalues 
    $$\lambda_k(-\calL)= \left(\tfrac{2\pi k\sigma}{\ell }\right)^2-\tfrac{2\pi kb}{\ell }i\, ,\quad  \lambda_k(\calL^*\calL)= \left(\tfrac{2\pi k\sigma}{\ell }\right)^4+\left(\tfrac{2\pi kb}{\ell }\right)^2\, .$$ 
    Hence $\gap(\calL)=({2\pi}\sigma/\ell )^2$ and $\sing(\calL)= \sqrt{(2\pi\sigma/\ell )^4+(2\pi b/\ell )^2}$. For $\sigma>0$, the spectral gap is of diffusive order $\Theta(\ell^{-2})$ and is not affected by the drift, whereas the singular value gap $\sing (\calL)$ is affected by the drift, and is of the order $\Theta(\ell^{-1})$ for $b\neq 0$ and of order $\Theta(\ell^{-2})$ otherwise. For a deterministic rotation, i.e.\ for $\sigma =0$, $\gap (\calL)=0$ and $\lambda (\calL)=0$, whereas
    $\sing(\calL)= 2\pi b/\ell$. In this case, $t_\rel(P) = \infty$, but $\overline t_\rel(P) < \infty$. 
\end{example}
\begin{example}[Random walks]\label{ex:randomwalk}
    We consider the discrete analogue to \Cref{ex:rotation}, i.e.\ a continuous-time asymmetric random walk on $S = \Z_n$ with generator $$\calL f(x) = \sigma^2\bigl(f(x+1)-2f(x)+f(x-1)\bigr) + b\bigl(f(x+1)-f(x)\bigr)$$ with coefficients $\sigma\in[0,\infty)$ and $b\in[-\sigma^2,\infty)$, and invariant measure $\mu = \unif(S)$. The eigenfunctions of $-\calL$ and $\calL^*\calL$ are again the trigonometric functions $\exp(2\pi i kx/n)$, $k\in\{0,\dots,n-1\}$, with eigenvalues
    \begin{align*}
        \lambda_k(-\calL)\ &=\ \left(4\sigma^2+2b\right)\sin^2\left(\tfrac{\pi k}{n}\right) - ib\sin\left(\tfrac{2\pi k}{n}\right)\,,\qquad\text{and}\\
        \lambda_k(\calL^*\calL)\ &=\ 16\sigma^2(\sigma^2+b)\sin^4\left(\tfrac{\pi k}{n}\right) + 4b^2\sin^2\left(\tfrac{\pi k}{n}\right)\,.
    \end{align*}
    Hence
    $\
        \gap(\calL) = (4\sigma^2+2b)\sin^2\left(\tfrac{\pi}{n}\right)$, \ and 
        $\ \sing(\calL) = 2\sqrt{4\sigma^2(\sigma^2+b)\sin^4\left(\tfrac{\pi}{n}\right) + b^2\sin^2\left(\tfrac{\pi}{n}\right)}$.
    Since $2\sigma^2+b\geq0$, the spectral gap is again of diffusive order $\Theta(n^{-2})$. However, in contrast to the diffusion case of \Cref{ex:rotation}, the spectral gap is affected by the drift, and is even non-zero for $\sigma = 0$ and $b>0$. The singular value gap is of the order $\Theta(n^{-1})$ for $b\neq 0$, and again of order $\Theta(n^{-2})$ if $b=0$.
\end{example}

While the above examples show that the spectral gap may be smaller than the singular value gap, the converse can also happen, for instance if the noise is transferred sequentially through the state space. This does not necessarily affect the asymptotic decay rate given by the spectral gap, but it is captured by non-asymptotic relaxation times and the singular value gap, as illustrated by the following example from \cite[Section 5.1]{Xu2026SVG}.

\begin{example}
    Consider the continuous-time Markov chain on $\{-1,+1\}^n$ with transitions
    $
        (x_1,\dots,x_{n-1},x_n)\to(x_2,\dots,x_n,Z)$ where $Z\sim\unif(\{-1,+1\})$. Its unique invariant probability measure is $\mu = \unif(\{-1,+1\}^n)$, and the transition matrix $P = \calL+I$ satisfies $P^n|_{L^2_0(\mu)} = 0$, since the chain reaches stationarity after $n$ transitions. Therefore, $\spec(P|_{L_0^2(\mu)}) = \{0\}$, which yields $\spec(-\calL|_{L_0^2(\mu)}) = \{1\}$ and $\gap(\calL) = 1$. However, for $f(x) = \sum_{k=1}^nkx_k$, we have $\calL f(x) = -\sum_{k=1}^nx_k$, $\norm{f}^2 = \sum_{k=1}^nk^2$ and  $\norm{\calL f}^2 = n$, so
        $$ \sing(\calL)\ \leq\ {\norm{\calL f}} /{\norm{f}} \ \in\ O({1}/{n}).$$
\end{example}

The singular value gap provides sharp bounds on the non-asymptotic quantity $\overline t_\rel(P)$. 

\begin{theorem}\label{thm:trel_svgap}
    It holds
    \begin{equation*}
        \frac{2(1-e^{-1})}{\sing(\calL)}\ \leq \ \overline t_\rel(P)\ \leq \ \frac{2e}{\sing(\calL)} \quad\text{ and }\quad
        \frac{1-e^{-1}}{\sing(\calL)}\ \leq \ t_\rel(P)\,.
    \end{equation*}
    Moreover,
    \begin{equation}\label{decaybarPt}
        \| \overline P_tf\|\ \le\ \frac{2}{t\cdot \sing(\calL)}\| f\|\quad\text{ for all }f\in L^2_0(\mu )\text{ and }t>0\,.
    \end{equation}
\end{theorem}

A closely related result has been shown in \cite{Chat2023Spectral} and in \cite{Xu2026SVG}, where also the relevance of the singular value gap for bounds of ergodic averages is discussed in detail. A very simple proof of \Cref{thm:trel_svgap} is included in \Cref{ssec:gapboundsproofs} below.

\begin{remark}[Lower bounds on relaxation times]
Note that \Cref{thm:trel_svgap} provides an easy way to derive lower bounds on relaxation times of non-reversible Markov processes. Indeed, by the variational characterisation of the singular value gap in \eqref{defsvgap}, it suffices to identify a function $f\in\mathcal D_0$ such that $\| \calL f\|/\| f\|$ is small. This is a replacement for the common approach of proving lower bounds for relaxation times of reversible Markov processes by
identifying a function $f\in\mathcal D_0$ such that $\mathcal E^\calL_1(f,f)/\| f\|^2$ is small.    
\end{remark}

We stress that although the decay of $ \| \overline P_tf\|$ in \eqref{decaybarPt} is only of order $O(1/t)$, an exponential decay can be obtained by considering modified time averages. Indeed, note that for $n\in\mathbb N$,
\begin{equation}\label{barPtn}
    \overline P_t^nf  \ =\ \Ex\left [P_{U_1+\dots+U_n}f\right]\ =\ \int_0^{nt}P_sf\ \nu_t^{*n}(\diff s) 
\end{equation}
with $\nu_t=\mathrm{Unif}(0,t)$ and i.i.d.\ random variables $U_1,\ldots ,U_n\sim\nu_t$.
Therefore, $\overline P_t^nf$ is a weighted time average. The weight distribution $\nu_t^{*n}$ has mean $\frac{nt}{2}$ and standard deviation ${\sqrt{n/12}\cdot t}$. For large $n$, by the central limit theorem, it is close to a normal distribution.
\begin{corollary}
    For any $t\geq\overline t_\rel(P)$, we have
    \begin{equation*}
        \norm{\overline P_t^nf}\leq e^{-n}\norm{f}\qquad\text{for all }f\in L_0^2(\mu)\text{ and }n\in\N\,.
    \end{equation*}
\end{corollary}
Thus to obtain a bound $\norm{\overline P_t^nf}\leq\varepsilon\norm{f}$ one can choose $n=\lceil\log(\varepsilon^{-1})\rceil$, resulting in a complexity $nt = t\cdot\lceil\log(\varepsilon^{-1})\rceil$ for any $t\geq\overline t_\rel(P)$. In applications to sampling with a given precision $\varepsilon$, one can first fix $n$ and then increase $t$ until relaxation is observed empirically.

\subsection{Singular value gap of the two-point motion}\label{ssec:twopointmotion}

\Cref{thm:trel_svgap} shows that the relaxation time of the time-averaged semigroup is equivalent to the
inverse singular value gap of the generator. However,
in general, the singular value gap is not sufficient to control the spectral gap and the relaxation time of the non-averaged semigroup, see \Cref{ex:rotation,ex:randomwalk}. 
The following results show that remarkably, the inverse singular value gap of the two-point motion is in fact equivalent to the relaxation time of the non-averaged semigroup, and yields a lower bound on the spectral gap of its generator. The proofs are given in \Cref{ssec:twopointproofs}.\medskip

Let $(P_t^{(2)})_{t\geq 0} = (P_t\otimes P_t)_{t\geq 0}$ on $L^2(S\times S,\mu\otimes\mu)$ denote the product semigroup, i.e., 
\begin{equation*}
    P_t^{(2)}(g\otimes h) = (P_tg)\otimes(P_th)\qquad\text{for all }g,h\in L^2(S,\mu)\text{ and }t\geq 0\,.
\end{equation*}
If $(P_t)_{t\geq 0}$ is the transition semigroup of a 
Markov process $((X_t)_{t\geq 0},(\P_x)_{x\in S})$ then 
\begin{equation*}
    (P_t^{(2)}f)(x,y)  =  \Ex_{x,y}[f(X_t,Y_t)]
\end{equation*}
is the transition semigroup
of two independent copies $((X_t,Y_t),\P_{x,y})$ of $(X_t)$, where $\P_{x,y} = \P_x\otimes\P_y$.
The generator $(\calL^{(2)},\dom(\calL^{(2)}))$ of $(P_t^{(2)})_{t\geq 0}$ satisfies
\begin{equation*}
    \calL^{(2)}f =  (\calL\otimes I+I\otimes \calL)f
\end{equation*}
for all $f\in L^2(S\times S,\mu\otimes\mu)$ such that $f(\cdot,y)\in \dom(\calL)$ for $\mu$-a.e.\ $y\in S$, $f(x,\cdot)\in \dom(\calL)$ for $\mu$-a.e.\ $x\in S$, and $(\calL\otimes I)f,(I\otimes \calL)f\in L^2(S\times S,\mu\otimes\mu)$. In particular,
\begin{equation*}
    \calL^{(2)}(g\otimes h)\ = \ (\calL g)\otimes h + g\otimes (\calL h)\qquad\text{for all }g,h\in \dom(\calL)\,.
\end{equation*}

In the reversible case,
\begin{equation}\label{gapsandsvgaps}
    \sing(\calL) = \gap(\calL) = \gap(\calL^{(2)}) = \sing(\calL^{(2)})\, .
\end{equation}
In general, these quantities do not agree, but $\sing(\calL^{(2)})$ provides a lower bound for $\gap(\calL)$.

\begin{theorem}\label{thm:L_L2}
   It holds
    \begin{equation*}
        \gap(\calL)\ \geq\ \frac{1}{2}\sing(\calL^{(2)})\,.
    \end{equation*}
\end{theorem}

Next, we go a step further and relate the non-asymptotic relaxation times to
the singular value gap of the two-point motion. We first remark that since
\begin{equation*}
    \norm{P_t^{(2)}}_{L_0^2(\mu\otimes\mu)\to L_0^2(\mu\otimes\mu)} \ = \ \norm{P_t}_{L_0^2(\mu)\to L_0^2(\mu)}\qquad\text{for all }t\geq 0\,,
\end{equation*}
the relaxation times of the semigroups $(P_t^{(2)})_{t\geq 0}$ and $(P_t)_{t\geq 0}$ coincide, and the inverse of $\sing(\calL^{(2)})$ also provides a lower bound for $t_\rel(P)$ by \Cref{thm:trel_svgap}.
To derive a matching upper bound, 
we make use of the following theorem by Wei \cite[Theorem 1.3]{Wei2021Diffusion}. For the convenience of the reader, we include a self-contained proof of this result in Appendix~\ref{appendix}.

\begin{theorem}[Quantitative Gearhart-Prüss theorem]\label{thm:wei}
    Let $(H,\dom(H))$ be a densely defined maximally dissipative operator on a Hilbert space. Then $H$ generates a strongly continuous contraction semigroup $(Q_t)_{t\geq 0}$ satisfying
    \begin{equation*}
        \norm{Q_t}\ \leq \ \exp\bigl({-t\Psi(H)+\tfrac{\pi}{2}}\bigr)\qquad\text{for all }t\geq 0\,,
    \end{equation*}
    where
    \begin{equation*}
        \Psi(H)\ =\ \inf\bigl\{\norm{(H-i\lambda)f}:f\in\dom(H),\,\lambda\in\R,\,\norm{f}=1\bigr\}\,.
    \end{equation*}
\end{theorem}
We will show in \Cref{ssec:twopointproofs} below that
\begin{equation}\label{eq:psibound}
    \Psi(\calL) \ \geq \ \frac{1}{2}\sing(\calL^{(2)})\, .
\end{equation}
Since $(P_t)_{t\ge 0}$ is a strongly continuous contraction semigroup, its generator $(\calL,\dom(\calL))$ is maximally dissipative.
Therefore, by Wei's Theorem, 
\begin{equation}\label{eq:twopointupper}
        \norm{P_t}_{L_0^2(\mu)\to L_0^2(\mu)} \ \leq \ \exp\left(-\frac{t\sing(\calL^{(2)})}{2}+\frac{\pi}{2}\right)\,.
    \end{equation}
As a consequence, we obtain the following result.

\begin{theorem}\label{cor:twopoint}
    We have
    \begin{equation*}
         \frac{1-e^{-1}}{\sing(\calL^{(2)})}\ \leq\ t_{\rel}(P)\ \leq\ \frac{\pi+2}{\sing(\calL^{(2)})}\,.
    \end{equation*}
    More generally, for every $\varepsilon\in(0,1)$,
    \begin{equation*}
    t_\rel(P,\varepsilon)\ \leq\ \frac{2}{\sing(\calL^{(2)})}\left(\frac{\pi}{2}+\log\frac1\varepsilon\right),
    \end{equation*}
    where
    \begin{equation*}
        t_\rel(P,\varepsilon)=\inf\bigl\{t\geq0:\norm{P_t}_{L^2_0(\mu)\to L^2_0(\mu)}\leq\varepsilon\bigr\}\,.
    \end{equation*}
\end{theorem}

\begin{remark}
    In the first version of this paper we have proven a weaker result that is
    based on a modified relaxation time, and we have stated the full result as an open problem. We thank an anonymous reader for pointing us to Wei's theorem which completes the proof of the full statement.
\end{remark}

As a consequence of our results, the relaxation time of a
non-reversible Markov process is equivalent to the averaged relaxation time of its two-point motion.

\begin{corollary}
It holds
    \begin{equation*}
         \gap(\calL)\ \geq\ \left(1-1/e\right)/\overline t_\rel(P^{(2)})\, ,\quad\text{and}\quad t_\rel(P)\ \le\ \tfrac{1+\pi /2}{1-1/e }\; \overline t_\rel(P^{(2)})\, .    
    \end{equation*}
\end{corollary}
These bounds follow directly by combining \Cref{thm:trel_svgap,thm:L_L2,cor:twopoint}.

\begin{example}[Rotation on the circle]\label{ex:rotation2}
Consider the case of a deterministic rotation on the circle, i.e., the setup of \Cref{ex:rotation} with $\sigma =0$ and $b\neq 0$. Then $1/\sing (\calL)<\infty $ but $ t_\rel(P)=\infty$. On the other hand, 
$\calL^{(2)} = b\frac{\diff}{\diff x} + b\frac{\diff}{\diff y}$ does not have a singular value gap since $\calL^{(2)}f = 0$ for $f(x,y) = g(x-y)$ with $g\in C^\infty(S^1_\ell)$. Thus in this case, the correct lower bound for $ t_\rel(P)$ is obtained by \Cref{cor:twopoint} but not by a direct application of the lower bound in \Cref{thm:trel_svgap}.
\end{example}

\begin{example}[Diffusion on the circle]\label{ex:rotation3}
    More generally, in the setup of \Cref{ex:rotation}, the eigenvalues of $-\calL^{(2)}$ and $(\calL^{(2)})^*\calL^{(2)}$ are
    \begin{eqnarray*}
        \lambda_{j,k}(-\calL^{(2)})&=& \left({2\pi \sigma}/{\ell }\right)^2(j^2+k^2)+({2\pi b}/{\ell })(j+k)i\, ,\\ \lambda_{j,k}((\calL^{(2)})^*\calL^{(2)})&=& \left({2\pi \sigma}/{\ell }\right)^4(j^2+k^2)^2+\left({2\pi b}/{\ell }\right)^2(j+k)^2\, ,
    \end{eqnarray*}
    with $j,k\in \mathbb Z$. Hence for all values of $b$, $$\sing(\calL^{(2)})=\min\bigl( 2(2\pi\sigma/\ell)^4,{(2\pi\sigma/\ell)^4+(2\pi b/\ell)^2}\big)^{1/2}\in [(2\pi\sigma/\ell)^2, 2(2\pi\sigma/\ell)^2]\,,$$
    so that $t_\rel(P)$ is of the diffusive order $\Theta(\ell^2/\sigma^2)$.
\end{example}

\subsection{Two-component approach for upper bounds on relaxation times}\label{ssec:twoscale}

The above results show that upper bounds on relaxation times of $(P_t)_{t\geq 0}$ can be obtained in terms of the inverse singular value gap or the inverse Poincaré constant of the associated generator $(\calL,\dom(\calL))$. In practice, however, inequalities such as
\begin{equation*}
    \E_i^\calL(f,f)\ \geq\ m_i\norm{f}^2
\end{equation*}
for $i\in\{1,2\}$ can often only be verified easily for functions in a subspace $\D_0^l$ of $\D_0$, or may even not hold on all of $\D_0$ because the noise is degenerate. 
Despite such partial dissipativity, relaxation can be observed in many cases. 

In order to deal with such situations, we consider an orthogonal decomposition
\begin{equation*}
    L_0^2(\mu)\ =\ \bbH^l\ \oplus\ \bbH^h
\end{equation*}
into two closed subspaces $\bbH^l$ and $\bbH^h$. Let
\begin{equation*}
    \Pi\colon L_0^2(\mu)\to \bbH^l
\end{equation*}
denote the associated orthogonal projection. We set $\D_0^l = \D_0\cap\bbH^l$ and $\D_0^h = \D_0\cap\bbH^h$, and assume that $$\Pi (\D_0 )\ \subseteq\ \D_0\, .$$ Since $\D_0$ is dense in $L_0^2(\mu)$, this implies that
$\D_0^l \supseteq \Pi (\D_0 )$ is dense in $\bbH^l$. The restrictions of the quadratic forms $(\E_i^\calL,\dom(\E_i^\calL))$ to $\bbH^l$ are again densely defined closed quadratic forms, and hence associated to 
self-adjoint linear operators on $\bbH^l$.

\begin{definition}[First- and second-order collapse]\label{def:collapse}
    For $i\in\{1,2\}$, the non-positive definite self-adjoint linear operators $({\calC}_i,\dom({\calC}_i))$ on $\bbH^l$ defined by 
    \begin{eqnarray*}
        \langle f,-{\calC}_ig\rangle &=& \E_i^\calL(f,g)\qquad\text{for all }f\in\dom(\E_i^\calL)\cap\bbH^l\text{ and }g\in\dom({\calC}_i)\,  ,\\
        \dom({\calC}_i) &=& \left\{  g\in \bbH^l:\exists\, h\in \bbH^l\text{ s.t.\ }\E_i^\calL(f,g)=\langle f,h\rangle \text{ for all }f\in\dom(\E_i^\calL)\cap\bbH^l\right\}
    \end{eqnarray*}
    are called the
    \emph{symmetric first- and second-order collapse of $\calL$}.
\end{definition}
In particular, it holds
\begin{align*}
   {\calC}_1g \ &=\ \Pi (\calL g+\calL^*g)/2&&\hspace{-8ex}\text{for } g\in\dom(\calL)\cap\dom(\calL^*)\cap \bbH^l,\text{ and}\\
   {\calC}_2g \ &=\ -\Pi\calL^*\calL g &&\hspace{-8ex}\text{for } g\in\dom(\calL^*\calL)\cap \bbH^l . 
\end{align*}
One could similarly define the non-symmetric first-order collapse of $\calL$ by replacing $\E_1^\calL$ in \Cref{def:collapse} with the (potentially non-symmetric) Dirichlet form $\E(f,g) = \langle f,-\calL g\rangle$. Our approach to bounding singular value gaps and relaxation times, however, is based only on the symmetric first- and second-order collapse.

\begin{example}[Setup for kinetic models]\label{ex:kinetic}
    In most examples that we consider, the state space is a product space
    $$S\ =\ X\times V\ =\ \{ (x,v): x\in X, v\in V\}\, ,$$
    and we define $\bbH^l$ as the linear subspace of functions in $L^2_0(S,\mu )$ that only depend on the first component, i.e.,
    $$\bbH^l\ =\ \left\{ g\circ\pi : g\in L^2_0(X,\nu )\right\}\ \cong \ L^2_0(X,\nu ) \, ,$$
    where $\pi (x,v)=x$, and $\mu (\diff x\diff v)=\nu (\diff x)\, \kappa (x,\diff v)$ is the disintegration of $\mu $ into a probability measure $\nu$ on $X$ and a stochastic kernel $\kappa$ from $X$ to $V$. In this case, 
    $$(\Pi f)(x)\ =\ \int f(x,v)\, \kappa (x,\diff v)$$
    is the conditional expectation of $f$ given $x$. Now suppose that $\calC_2$ is the generator of a reversible diffusion process with state space $X$ and invariant measure $\nu$. Then in the terminology introduced in \cite{EberleLoerler2024Lifts}, the Markov process with generator $\calL$ is called a \emph{second-order lift} of the
    reversible diffusion if for all $f,g\in \dom({\calC_2})$,
    $$\langle f\circ\pi ,\calL (g\circ\pi )\rangle =0\, \quad\text{and}\quad\, \frac 12\langle \calL (f\circ\pi ),\calL (g\circ\pi )\rangle =\langle f,-\calC_2 g\rangle\, .$$ 
    Thus in the case of a second-order lift, the first-order collapse of the generator $\calL$ vanishes and the second-order collapse coincides (up to a factor $2$) with $\calC_2$. The setup considered here is more general. It allows for processes with non-vanishing first-order collapse, and the second-order collapse is not necessarily the generator of a Markov process, see the examples in  \Cref{sec:Examples} below.
\end{example}

\begin{theorem}[Lifted Poincaré inequality]\label{thm:liftedPoincare}
    Let $(\E,\dom(\E))$ be a densely defined, closed, non-negative definite symmetric bilinear form on $L_0^2(\mu)$, and let $({\calC},\dom({\calC}))$ be the non-positive self-adjoint linear operator on $\bbH^l$ that is the generator of $\E$ restricted to $\bbH^l$.
    Assume that there are constants $m,C\in (0,\infty)$ with
    \begin{align}\tag{B}\label{eq:B}
        \E(g,g)\ &\geq\ m\norm{g}^2&&\kern-2em\text{for all }g\in \dom(\E)\cap\bbH^l\,,\\
        \E(f,g)\ &\leq\ C\norm{f}\norm{{\calC}g}&&\kern-2em\text{for all }f\in\dom(\E)\cap\bbH^h,\, g\in\dom({\calS})\,.\tag{C}\label{eq:C}
    \end{align}
    Then
    \begin{equation}
        \norm{f}^2\ \leq \ \frac{2}{m}\E(f,f)\ +\ (1+2C^2)\norm{f-\Pi f}^2
    \end{equation}
    for all $f\in\dom(\E)$ with $\Pi f\in\dom(\E)$.
\end{theorem}

This theorem can be applied to prove Poincar\'e inequalities or lower bounds on the singular value gap for generators $\calL$ that satisfy a partial dissipativity on $\bbH^h$, i.e.\ 
\begin{equation}\label{eq:A1}\tag{A1}
    \E_1^\calL(f,f)\ \geq \ \gamma\norm{f-\Pi f}^2\qquad\text{for all }f\in\D_0
\end{equation}
or
\begin{equation}\label{eq:A2}\tag{A2}
    \E_2^\calL(f,f)\ \geq \ \gamma^2\norm{f-\Pi f}^2\qquad\text{for all }f\in\D_0
\end{equation}
for some $\gamma\in (0,\infty)$. If $\calL$ restricted to $\bbH^l$ satisfies a Poincaré inequality or has a singular value gap, that is if
there exists a constant $m_1\in (0,\infty)$ such that
\begin{equation}\label{eq:B1}\tag{B1}
    \E_1^\calL(f,f)\ \geq \ m_1\norm{f}^2\qquad\text{for all }f\in\D_0^l\,,
\end{equation}
or there exists a constant $m_2\in (0,\infty)$ such that
\begin{equation}\label{eq:B2}\tag{B2}
    \E_2^\calL(f,f)\ \geq \ m_2\norm{f}^2\qquad\text{for all }f\in\D_0^l\,,
\end{equation}
respectively, 
a Poincaré inequality or singular value gap can be obtained for $\calL$ under an additional assumption on the interaction 
of $\bbH^l$ and $\bbH^h$ under $\calL$. Specifically, setting 
\begin{equation*}
    \calL \ = \ {\calA} \ + \ \gamma(\Pi-I)
\end{equation*}
yields a non positive definite linear operator $({\calA},\dom(\calL))$, and the additional assumption is that
there exists a constant $C_1\in(0,\infty)$ such that
\begin{equation}\label{eq:C1}\tag{C1}
    \E_1^{\calA}(f,g)\ \leq\ C_1\norm{f}\norm{{\calC}_1g}\qquad\text{for all }f\in\D_0^h,\, g\in\dom({\calC}_1)\,,
\end{equation}
or that there exists a constant $C_2\in(0,\infty)$ such that
\begin{equation}\label{eq:C2}\tag{C2}
    \E_2^{\calA}(f,g)\ \leq\ C_2\norm{f}\norm{{\calC}_2g}\qquad\text{for all }f\in\D_0^h,\, g\in\dom({\calC}_2)\,,
\end{equation}
respectively. 
In practice, if $\calA g\in\dom(\calA^*)$ for all $g\in\dom(\calC_2)$, as is for instance the case in discrete models or typically in the absence of boundary conditions, \eqref{eq:C2} can often be verified by checking the equivalent condition
\begin{equation}\tag{C2'}\label{eq:C2'}
    \norm{(I-\Pi)\calA^*\calA g}\ \leq \ C_2\norm{\calC_2g}\qquad \text{for all }g\in\dom(\calC_2)\,.
\end{equation}

The most interesting case arises when $\calL$ is dissipative on $\bbH^h$ and has a singular value gap on $\bbH^l$, i.e., if \eqref{eq:A1} and \eqref{eq:B2} are satisfied. In this case, we obtain the following.

\begin{theorem}\label{thm:FPI}
    If  \eqref{eq:A1}, \eqref{eq:B2} and \eqref{eq:C2} are satisfied, then
    \begin{equation*}
        \norm{\overline P_tf} \ \leq \ \max\biggl(\frac{2}{m_2},\frac{1+2C_2^2}{\gamma^2}\biggr)^{1/2}\norm{\calL\overline P_tf} \ + \ \frac{4\gamma}{t}\max\biggl(\frac{2}{m_2},\frac{1+2C_2^2}{\gamma^2}\biggr)\norm{f-P_tf}
    \end{equation*}
    for all $f\in L_0^2(\mu)$ and $t>0$. In particular,
    \begin{equation*}
        \overline t_\rel(P)\ \leq \  16e\max\left(\frac{2\gamma}{m_2},\frac{1+2C_2^2}{\gamma}\right)\,,\text{ and }\
        \sing(\calL )\ \ge \ \frac{1-e^{-1}}{8e}\min\left(\frac{m_2}{2\gamma},\frac{\gamma}{1+2C_2^2}\right)\, .
    \end{equation*}
\end{theorem}

    In the setup of second-order lifts introduced in \cite{EberleLoerler2024Lifts}, a related result, termed a \emph{flow Poincaré inequality}, has been shown in \cite{EGHLM2025Lifting}, essentially under the assumptions \eqref{eq:A1}, \eqref{eq:B2} and \eqref{eq:C2}. While these are also considered here, the new result is far more generally applicable, since it does not require the first-order collapse to vanish, and applies directly to discrete or continuous models. Furthermore, the proof of \Cref{thm:FPI} is simpler than that of the flow Poincaré inequality, which relies on a space-time divergence lemma \cite{EberleLoerler2024Divergence}.\medskip
    
  A direct application of \Cref{thm:FPI} only yields bounds on $\overline{t}_\rel$, whereas the flow Poincaré inequality yields bounds for non-averaged relaxation times. However, by considering the generator of the associated two-point motion, we obtain the following result as a direct consequence of \Cref{thm:FPI} combined with \Cref{thm:L_L2,cor:twopoint}.

\begin{corollary}\label{cor:FPI}
    If  \eqref{eq:A1}, \eqref{eq:B2} and \eqref{eq:C2} are satisfied for $\calL^{(2)}$, then for all $\varepsilon >0$,
    \begin{eqnarray*}
        \gap (\calL ) & \ge & \frac{1-e^{-1}}{16e}\min\left(\frac{m_2}{2\gamma},\frac{\gamma}{1+2C_2^2}\right)\, ,\qquad\qquad\qquad\text{and}\\
        t_\rel(P,\varepsilon)& \leq & \frac{16e}{1-e^{-1}}\left( \frac \pi 2+\log \left(\frac{1}{\varepsilon}\right)\right)\, \max\left(\frac{2\gamma}{m_2},\frac{1+2C_2^2}{\gamma}\right)\,.
    \end{eqnarray*}
\end{corollary}

\Cref{thm:liftedPoincare} also immediately yields the following results.

\begin{corollary}\label{cor:LPI}
    \begin{enumerate}[(i)]
        \item If \eqref{eq:A1}, \eqref{eq:B1} and \eqref{eq:C1} are satisfied, then
        \begin{equation*}
            \norm{f}^2\ \leq \ \left(\frac{2}{m_1}+\frac{1+2C_1^2}{\gamma}\right)\E_1^\calL(f,f)
        \end{equation*}
        for all $f\in\dom(\E_1^\calL)$ with $\Pi f\in\dom(\E_1^\calL)$. In particular,
        \begin{equation*}
            t_\rel(P)\ \leq \ \frac{2}{m_1}+\frac{1+2C_1^2}{\gamma}\,.
        \end{equation*}

        \item If  \eqref{eq:A2}, \eqref{eq:B2} and \eqref{eq:C2} are satisfied, then
        \begin{equation*}
            \norm{f}^2\ \leq \ \left(\frac{2}{m_2}+\frac{1+2C_2^2}{\gamma^2}\right)\E_2^\calL(f,f)
        \end{equation*}
        for all $f\in\dom(\E_2^\calL)$ with $\Pi f\in\dom(\E_2^\calL)$. In particular,
        \begin{equation*}
            \overline t_\rel(P)\ \leq \ 2e\left(\frac{2}{m_2}+\frac{1+2C_1^2}{\gamma^2}\right)\,.
        \end{equation*}
    \end{enumerate}
\end{corollary}

\section{Examples}\label{sec:Examples}

\subsection{Lifted random walks}\label{sec:liftedRW}

Markov chain Monte Carlo methods rely on constructing a Markov chain with a given target invariant measure, with their efficiency determined by the speed of convergence of the constructed chain. While reversibility of the Markov chain with respect to the target measure is a simple criterion ensuring its invariance, non-reversibility is crucial in avoiding the slow, diffusive behaviour inherent to most algorithms based on local, reversible dynamics \cite{Lelievre2013Optimal,Neal2004AsymptoticVariance,Hwang2005Accelerating,Krauth2024HMCvsECMC}. However, this leads to processes that are much harder to analyse, since many traditional tools fail to give sharp bounds or break down completely in the non-reversible case.

The lifted random walk on the discrete circle $\Z_n = \Z/(n\Z)$ was introduced by Diaconis, Holmes and Neal \cite{Diaconis2000Lift} as a simplified model for non-reversible Markov chains used in sampling \cite{DKPR1987HybridMonteCarlo,Horowitz1991GeneralizedHMC}, see also \cite{Chen1999Lift}. While the one-dimensional chain was analysed completely in \cite{Diaconis2000Lift}, the arguments based on explicit calculations do not directly apply to its natural generalisations on the discrete torus $\mathbb Z_n^d=\mathbb Z^d/(n\mathbb Z)^d$ or on general abelian groups, necessitating the development of a more broadly applicable approach. It was conjectured in \cite{DiaconisMiclo2013SecondOrder} that such lifted random walks, or second-order Markov chains in their terminology, can be used to speed up convergence of random walks on abelian groups by a square root. The two-component approach introduced in \Cref{ssec:twoscale}, combined with the associated two-point motion, allows us to prove this conjecture in \Cref{thm:LRWgroups} below by providing sharp upper and lower bounds on the relaxation times.  \medskip

\begin{figure}
    \centering
    \begin{subfigure}[b]{0.3\textwidth}
         \centering
         \includegraphics[width=\textwidth]{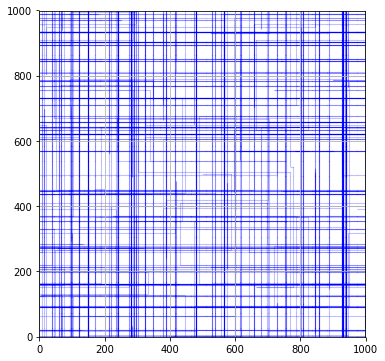}
         \caption{$\gamma = 0.001$}
         \label{fig:LRW1}
    \end{subfigure}
    \hfill
    \begin{subfigure}[b]{0.3\textwidth}
         \centering
         \includegraphics[width=\textwidth]{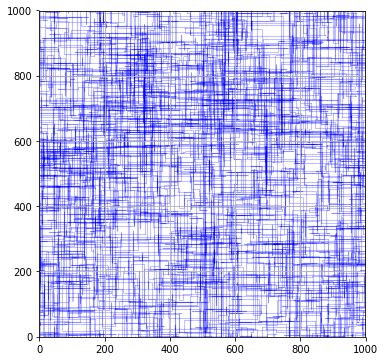}
         \caption{$\gamma = 0.02$}
         \label{fig:LRW2}
    \end{subfigure}
    \hfill
    \begin{subfigure}[b]{0.3\textwidth}
         \centering
         \includegraphics[width=\textwidth]{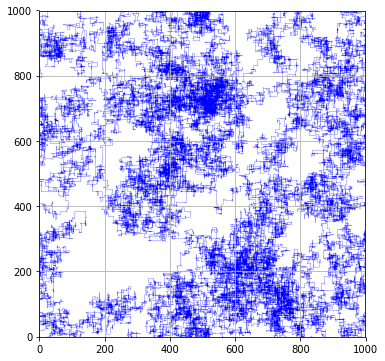}
         \caption{$\gamma = 0.2$}
         \label{fig:LRW3}
    \end{subfigure}
    \caption{Trajectories of the lifted random walk on $\Z_n^2$ with $n=1000$ and different choices of $\gamma$.}    \label{fig:LRW}
\end{figure}

Let $d,n\in\mathbb N$, $n\ge 2$.
The lifted random walk on $\mathbb Z_n^d=\mathbb Z^d/(n\mathbb Z)^d$ is a stochastic process that moves persistently in one coordinate direction between reorientations at rate $\gamma\in(0,\infty)$, at which a new direction is chosen uniformly at random. More precisely, we consider the continuous-time Markov chain with state space $S = \Z_n^d\times V$ where $V = \{\pm e_i:i\in\{1,\dots,d\}\}$, and with generator
\begin{equation*}
    \calL f(x,v)  =  f(x+v,v)-f(x,v) \ + \ \gamma (\Pi f-f)(x,v)\,,\quad
    \Pi f(x,v)= \frac{1}{2d}\sum_{w\in V}f(x,w)\,.
\end{equation*}
The unique invariant probability measure of the associated transition semigroup is the uniform distribution $\mu = \unif(S)$, and the adjoint of the generator in $L^2(\mu)$ is
\begin{equation*}
    \calL^*f(x,v)\ =\ f(x-v,v)-f(x,v)\ + \ \gamma(\Pi f-f)(x,v)\,.
\end{equation*}
Sample trajectories of the lifted random walk on $\Z_n^2$ for $n=1000$ and different choices of $\gamma$ are plotted in \Cref{fig:LRW}. For small values of $\gamma$, the persistent motion dominates, leading to the grid-like pattern in \Cref{fig:LRW1}, while for large values of $\gamma$, the frequent reorientations cause the diffusive behaviour to be seen in \Cref{fig:LRW3}. At least heuristically, the optimal value of $\gamma$ can be found in the middle, corresponding to \Cref{fig:LRW2}. We make these three regimes precise in \Cref{thm:LRW} below.

In order to apply the two-scale approach of \Cref{ssec:twoscale}, we set $\bbH^l = \Ran(\Pi)$.
In this setting, the generators of the quadratic forms $\E_i^\calL$, $i=1,2$ are
\begin{equation*}
    {\calS}_1 = (\calL+\calL^*)/2\qquad\text{and}\qquad {\calS}_2 = -\calL^*\calL\,.
\end{equation*}
In particular, the first- and second-order collapses ${\calC}_i$ are given by $\Pi {\calS}_i$ restricted to $\bbH^l$. Functions in $\bbH^l$ can be written as $f\circ\pi$ for a function $f\colon \Z_n^d\to\R$ and $\pi(x,v) = x$. We then have $\calL(f\circ\pi)(x,v) = f(x+v)-f(x)$, so that
\begin{align*}
    \calS_1(f\circ\pi)(x,v) &= \frac{1}{2}\big(f(x+v)-2f(x)+f(x-v)\big)\,, \\
    \calS_2(f\circ\pi)(x,v) &= f(x+v)-2f(x)+f(x-v) -  \frac{\gamma}{2d}\sum_{w\in V}\big(f(x+w)-f(x+v)\big)\,.
\end{align*}
Denoting $\Delta f(x)= \sum_{i=1}^d\big(f(x+e_i)-2f(x)+f(x-e_i)\big)$, this yields
\begin{align*}
    {\calC}_1(f\circ\pi) = \frac{1}{2d}(\Delta f)\circ\pi \qquad\text{and}\qquad {\calC}_2(f\circ\pi) = \frac{1}{d}(\Delta f)\circ\pi\,.
\end{align*}

\begin{theorem}\label{thm:LRW}
    There exist universal constants $c,C\in (0,\infty )$ such that for all $d,n\in\mathbb N$ with $n\ge 2$ and all $\gamma\in (0,\infty )$, the  lifted random walk on $\Z_n^d$ satisfies
     \begin{equation}
        c\cdot d\cdot\max\Big({n^2\min (1,\gamma )},{\gamma}^{-1}\Big)\ \le\  \tfrac{e-1}{e^2}\,\overline t_\rel\ \le\  t_\rel \ \leq \ C\cdot d\cdot\max\Big({n^2\min (1,\gamma )},{\gamma}^{-1}\Big) \,.
    \end{equation}
    Furthermore, 
    \begin{equation*}
        \gap(\calL ) \ \geq \  C^{-1}\cdot d^{-1}\cdot\min\Big({n^{-2}\max (1,\gamma^{-1} )},{\gamma}\Big)\,.
    \end{equation*}
\end{theorem}

The constants $c$ and $C$ are explicit. For example, we obtain 
\begin{equation}
    t_\rel \ \leq \ \frac{8e(\pi+2)}{1-e^{-1}}\max\bigl(dn^2\min(1,\gamma)/20,(9+16d)/\gamma\bigr)\,.
\end{equation}
\Cref{thm:LRW} shows that there are three regimes for the relaxation time: It is of order $dn^2$ for $\gamma\ge 1$, of order $\gamma dn^2$ for $\gamma\in [1/n,1]$, and of order $d/\gamma$ for $\gamma\in (0,1/n)$. In particular, $t_\rel$ is of order $dn$ for $\gamma\propto 1/n$.

The proof is given in \Cref{ssec:ProofLRW}.
The upper and lower bounds on $t_\rel$ are separate results.
The proof of the upper bound is based on the two-component approach developed in 
\Cref{ssec:twoscale} (in particular \Cref{thm:FPI}) applied to the associated two-point motion, see \Cref{lem:upperboundLRW,lem:LRWtwopoint}. The result shows that this approach can yield sharp bounds even if the constant $C_2$ is not of order $O(1)$. On the other hand,
the proof of the lower bound is based on the identification of an explicit mean zero function $f\colon S\times\mathbb Z_n^d\to\mathbb C$ such that $\|\calL f\|/\| f\|$ is small, see \Cref{lem:lowerboundLRW}. Remarkably, the same type of function $f$ yields a bound of the correct order for any value of $\gamma$, although there are different regimes for the order of the relaxation time depending on the size of $\gamma$. The function $f$ is carefully designed to take into account the interaction between the velocity and position variables. A straightforward choice for $f$ (for example a function depending only on $x$) would not yield bounds of the correct order.\medskip
\bigskip

In contrast to the lower bounds on the relaxation times given above, our technique of proof for
the upper bounds is much more generally applicable and not restricted to lifted random walks on
$\mathbb Z_n^d$. Other applications on discrete state spaces will be studied in follow-up work. For the moment, we remark that the arguments can almost one-to-one be carried over to lifted random walks on abelian groups.\medskip

Consider an abelian group $(G,+)$ and a symmetric set $V\subseteq G$ of generators. Let $\gamma\in (0,\infty)$.
The lifted random walk on $G$ is again the continuous-time Markov chain with state space $S = G\times V$ and generator
\begin{equation*}
    \calL f(x,v)\ = \ f(x+v,v)-f(x,v) \ + \ \gamma(\Pi f-f)(x,v)\,,\quad
    \Pi f(x,v) \ = \ \frac{1}{|V|}\sum_{w\in V}f(x,w).
\end{equation*}
The unique invariant probability measure is the uniform distribution $\mu = \nu\otimes\kappa$ with $\nu = \unif(G)$ and $\kappa = \unif(V)$, and the adjoint of the generator in $L^2(\mu)$ is
\begin{equation*}
    \calL^*f(x,v) \ = \ f(x-v,v)-f(x,v)\ +\ \gamma(\Pi f-f)(x,v)\,.
\end{equation*}
The first- and second-order collapses are given by
\begin{equation*}
    \calC_1 \ =\  \tfrac{1}{|V|}\Delta\qquad\text{and}\qquad \calC_2\ =\ \tfrac{2}{|V|}\Delta\, ,\qquad\text{where}
\end{equation*}
\begin{equation*}
    \Delta g(x)\ =\ \frac{1}{2}\sum_{v\in V}\bigl(g(x+v)-2g(x)+g(x-v)\bigr)\,,\qquad g\colon G\to\R\, .
\end{equation*}
We let $\gap(\Delta)$ denote the spectral gap of $\Delta$ in $L^2(\nu)$. With these notations, the arguments in the proof of the upper bound in \Cref{thm:LRW} carry through analogously on abelian groups, yielding the following.

\begin{theorem}\label{thm:LRWgroups}
        There exist universal constants $c,C\in (0,\infty)$ such that the lifted random walk on abelian groups satisfies
        \begin{align*}
            t_\rel(P)\ &\leq \ C |V|\max\left(\frac{\min(1,\gamma)}{\gap(\Delta)},\gamma^{-1}\right)\qquad\text{and}\\
            \gap(\calL)\ &\geq\ c |V|^{-1}\min\Bigl(\gap(\Delta) \max(1,\gamma^{-1}), \gamma\Bigr)\,.
        \end{align*}
\end{theorem}
\Cref{thm:LRWgroups} shows that for $\gamma = \sqrt{\gap(\Delta)}$ one has 
\begin{equation*}
    \gap(\calL)\ \geq\ c|V|^{-1}\sqrt{\gap(\Delta)}\qquad\text{and}\qquad t_\rel(P)\ \leq\ C|V|/\sqrt{\gap(\Delta)}
\end{equation*}
with universal constants $c,C\in (0,\infty)$.
This partially confirms Conjecture 1.3 in \cite{DiaconisMiclo2013SecondOrder}, which states that there exist $c_0,C_0\in (0,\infty)$, potentially depending on $|V|$, such that
\begin{equation*}
    c_0 \sqrt{\gap(\Delta)} \ \leq \ \gap(\calL)\ \leq \ C_0 \sqrt{\gap(\Delta)}\,.
\end{equation*}
The key statement is the lower bound on $\gap(\calL)$ which is proved by \Cref{thm:LRWgroups}. It is of particular interest, since it confirms a diffusive-to-ballistic speed-up of the lifted random walk compared to the simple random walk on $G$ with steps chosen uniformly from $V$. While we do not obtain a direct upper bound on $\gap(\calL)$, a non-asymptotic lower bound
\begin{equation*}
    t_\rel(P) \ \geq \ C_0^{-1}\sqrt{|V|/\gap(\Delta)}
\end{equation*}
for an absolute constant $C_0\in (0,\infty)$ follows from the lower bound for relaxation times of lifts, see \cite[Remark 12]{EberleLoerler2024Lifts}, showing that, in terms of the relaxation time, the square-root speed-up is optimal, as conjectured.

\subsection{Switching flows perturbed by noise}

Let $a_1,\ldots ,a_d\in (0,\infty )$. We consider a continuous-time Markov process with state space $S=X\times V$ and invariant probability measure $\mu (\diff x\diff v)=\nu (\diff x)\kappa (\diff v)$.  
We assume that $X$ is a $d$-dimensional torus, i.e.,
$X=\prod_{i=1}^d S^1_{a_i}$ with $  S^1_{a_i}= \mathbb R/(a_i\mathbb Z)$,  $\nu$ is the uniform distribution on $X$, and $(V,\mathcal B ,\kappa )$ is an arbitrary probability space. The $X$-component $(X_t)_{t\ge 0}$ of the process
satisfies an SDE 
$$\diff X_t\ =\ \beta (V_t)\diff t\, +\, \sigma\diff B_t\, ,$$
and the $V$-component $(V_t)_{t\ge 0}$ is a Markov process with generator $\gamma (\Pi -I)$ where
$$\Pi f(x,v)\ =\ \int f(x,w)\, \kappa (\diff w)\, . $$
In other words, $(X_t)_{t\ge 0} $ follows a randomly perturbed flow of a vector field $\beta (V_t)$ that is switched after random times which are independent and exponentially distributed with parameter $\gamma\in (0,\infty )$. The noise can be degenerate. 

More precisely, we assume that $(B_t)_{t\ge 0}$ is an $n$-dimensional Brownian motion for some $n\in\mathbb N$, $\sigma\in\mathbb R^{d\times n}$ is an arbitrary matrix, and $\beta $ is a measurable map in $L^4(V\to\mathbb R^d,\kappa )$ such that
\begin{equation}\label{eq:betacentered}
    \int \beta (v)\, \kappa (\diff v)\ =\ 0\, .
\end{equation}
Let $A=\sigma\sigma^T$ denote the possibly degenerate diffusion matrix, and let
\begin{eqnarray}
    \label{def:T} T& =& \int \beta (v)\otimes \beta (v)\, \kappa (\diff v)\, ,\\
    \label{def:Q} Q& =& \int \beta (v)\otimes \beta (v)\otimes \beta (v)\otimes \beta (v)\, \kappa (\diff v)\, ,
\end{eqnarray}
denote the tensors describing the second and fourth moments of $\beta$. The generator of the
Markov process on $X\times V$ is
$$\mathcal L\ =\ \mathcal L_x+\mathcal L_v+\beta (v)\cdot\nabla_x\, , \quad \mathcal L_x\ =\ \tfrac 12 A\mathbin{:} \nabla^2_x\,, \quad \mathcal L_v\ =\ \gamma (\Pi-I) \, .$$
Decomposing $\mathcal L=\calA+\gamma (\Pi -I)$, we have $\calA=\mathcal L_x+\beta (v)\cdot\nabla _x$ and $\calA^*=\mathcal L_x-\beta (v)\cdot\nabla _x$. Now suppose $g\colon X\to\mathbb R$ is a smooth function that we extend to $X\times V$ by setting $g(x,v):=g(x)$. Then we obtain
\begin{equation}\label{eq:AstarAflows}
     (\calA^*\calA g)(x,v)\ =\ \mathcal L_x^2g (x,v)-(\beta (v)\otimes\beta (v))\mathbin{:}\nabla^2_xg(x,v)\, ,
\end{equation}
and hence the first- and second-order collapse of $\mathcal L$ are given by
\begin{eqnarray}  \label{eq:S1flows} 
     \calC_1g &=&  \Pi (\calA^*g+\calA g)/2\ =\ \mathcal L_xg\ =\ \tfrac 12 A\mathbin{:}\nabla^2g\, ,\\
     \label{eq:S2flows} \calC_2g &=&  -\Pi \calA^*\calA g\ =\ -\mathcal L_x^2g+T\mathbin{:}\nabla^2g\ =\ -\tfrac 14\left( A\mathbin{:}\nabla^2 \right)^2g+T\mathbin{:}\nabla^2g\, .
\end{eqnarray}
The eigenfunctions of both operators are the trigonometric functions $\exp(ip\cdot x)$ with $p_k\in \tfrac{2\pi}{a_k}\mathbb Z$ for all $k=1,\ldots ,d$. The corresponding eigenvalues of $-\calC_1$ and $-\calC_2$ are
$$\lambda_1(p)\ =\ \tfrac 12p^TAp\, ,\quad \lambda_2(p)\ =\ \tfrac 14 \left( p^TAp\right)^2+p^TTp\, .$$
Correspondingly, the spectral gaps of the operators are the minima of the eigenvalues corresponding to $p=\tfrac{2\pi }{a_k}e_k$, $k=1,\ldots ,d$, i.e.\ 
\begin{eqnarray}\label{eq:m1m2flows}
    m_1\ =\ 2\pi^2\min_k \frac{A_{kk}}{a_k^2}\, ,\qquad m_2\ =\ 4\pi^2\min_k \left(\frac{T_{kk}}{a_k^2}+\pi^2\frac{A_{kk}}{a_k^4}\right)\, .
\end{eqnarray}
Whereas $m_1 $ only takes into account the diffusion matrix $A$, $m_2$ takes into account both the diffusion coefficients and the transport by the flow that is encoded in $T$. In particular, $m_2$ can be strictly positive even if both $A$ and $T$ are degenerate.

In order to derive bounds on the relaxation time based on the spectral gap of the second-order collapse, one has to verify Condition \eqref{eq:C2}. It turns out that this requires a bound of the form
\begin{equation}
    H\mathbin{:} (Q-T\otimes T)\mathbin{:} H\ \le\ C^2\cdot\bigl\| T^{1/2}HT^{1/2}\bigr\|_F^2\quad\text{for all symmetric matrices }H\in\mathbb R^{d\times d}.
\end{equation}
The smallest constant $C$ for which this bound holds is
\begin{equation*}
  C_{Q,T}:= \sup\left\{  \left.\sqrt{H\mathbin{:} (Q-T\otimes T)\mathbin{:} H}\, \right| H\in\mathbb R^{d\times d}\text{ symmetric,}\,\bigl\| T^{1/2}HT^{1/2}\bigr\|_F\le 1\right\} .
\end{equation*}
By applying \Cref{thm:FPI}, \Cref{cor:FPI} and \Cref{cor:LPI}, we obtain:
\begin{theorem}\label{thm:shear}
For the Markov process $(X_t,V_t)$ defined above, the following bounds hold:
\begin{align*}
    t_\rel\ & \leq\  \frac{2}{m_1}+\frac{3}{\gamma}\,   ,\qquad
    t_\rel \ \leq \  \frac{8e(\pi+2)}{1-e^{-1}}\,\max\left( \frac{2\gamma}{m_2},\frac{5+2C_{Q,T}^2}{\gamma}\right)\, , 
    \\
    \gap (\calL )\ & \geq\ \frac{1-e^{-1}}{16e}\min\left(\frac{m_2}{2\gamma},\frac{\gamma}{5+2C_{Q,T}^2}\right)\, .
\end{align*}
\end{theorem}
The proof is given in \Cref{ssec:Proofsflows}. The bounds show that both the transport by the flow and the diffusion improve relaxation. Apart from the first bound, they are of ballistic
order $O(\max a_k)$ if $\gamma$ is proportional to $\sqrt m_2$.

\subsection{Langevin dynamics}\label{ssec:Langevin}

Let $U\in C^2(\R^d)$ such that $\lim_{|x|\to\infty}U(x) = \infty$ and $\exp(-U)\in L^1(\R^d)$.
Langevin dynamics with friction $\gamma\in (0,\infty)$ is the diffusion process $(X_t,V_t)_{t\geq 0}$ with state space $\R^d\times\R^d$ solving the SDE
\begin{align*}
    \diff X_t \ &=\ V_t\diff t\,,\\
    \diff V_t \ &=\ -\nabla U(X_t)\diff t \,-\, \gamma V_t\diff t \,+\,\sqrt{2\gamma}\diff W_t\,.
\end{align*}
It is a paradigmatic example of a non-reversible Markov process. The study of convergence to its stationary measure
\begin{equation}\label{eq:mu}
    \mu(\diff x \diff v) = \nu(\diff x)\kappa(\diff v)\qquad\text{with}\qquad \nu(\diff x)\propto e^{-U(x)}\diff x\quad\text{and}\quad\kappa = \mathcal{N}(0,I_d)
\end{equation}
is challenging due to the interplay of non-reversibility and the degenerate noise. In fact, this process and the associated kinetic Fokker-Planck equation served as a major motivation for the development of many hypocoercivity techniques \cite{Villani2009Hypocoercivity,DMS2009Hypocoercivity,Eberle2019Langevin,Albritton2021Variational}. 

While upper bounds on the $L^2$-relaxation of the associated transition semigroup are by now well-established using the approach of space-time Poincaré inequalities and second-order lifts \cite{Albritton2021Variational,Cao2019Langevin,EberleLoerler2024Lifts} or the approach by Dolbeault, Mouhot and Schmeiser based on a modified $L^2$-norm \cite{DMS2015Hypocoercivity,FanLiLu2026SharpL2}, these strongly make use of the fact that the first-order collapse of the associated generator vanishes. Here, we provide a simpler proof of these known results based on the two-component approach of \Cref{ssec:twoscale}. Since this does not require the first-order collapse to vanish, our result also applies to Langevin dynamics perturbed by a diffusion in the position variable, which lay beyond the scope of previous hypocoercivity approaches.
Furthermore, while the known upper bounds on the relaxation time are sharp at the critical friction $\gamma_*$ \cite{EberleLoerler2024Lifts}, we establish simple lower bounds on the relaxation time in both the underdamped regime $\gamma\in (0,\gamma_*)$ and the overdamped regime $\gamma\in(\gamma_*,\infty)$ by proving upper bounds on the singular value gap using explicit test functions. \smallskip

We consider the diffusion $(X_t,V_t)_{t\geq 0}$ with state space $\R^d\times\R^d$ given by
\begin{align*}
    \diff X_t \ &=\ V_t\diff t\, -\, \alpha \nabla U(X_t)\diff t \,+\, \sqrt{2\alpha}\diff B_t\\
    \diff V_t \ &=\ -\nabla U(X_t)\diff t \,-\, \gamma V_t\diff t \,+\,\sqrt{2\gamma}\diff W_t\,,
\end{align*}
where $(B_t)_{t\geq 0}$ and $(W_t)_{t\geq 0}$ are independent $d$-dimensional Brownian motions, and $\alpha,\gamma\in[0,\infty)$. This defines a Markov process with invariant probability measure given by \eqref{eq:mu}
and generator $\calL = \calH + \alpha\calL_x + \gamma\calL_v$
with
\begin{align*}
    \calH\ &= \ v\cdot\nabla_x-\nabla U(x)\cdot\nabla_v\,,\\
    \calL_x\ &= \ -\nabla U(x)\cdot\nabla_x + \Delta_x\,,\\
    \calL_v\ &= \ -v\cdot\nabla_v+\Delta_v\,.
\end{align*}
Here $\calH$ is the generator of the Hamiltonian flow associated to the Hamiltonian $H(x,v) = U(x)+\frac{1}{2}|v|^2$, $\calL_x$ is the generator of an overdamped Langevin diffusion with invariant measure $\nu$ in the $x$-variable, and $\calL_v$ is the generator of an Ornstein-Uhlenbeck process with invariant measure $\kappa$ in the $v$-variable. The case $\alpha = 0$ yields the usual (kinetic) Langevin dynamics, so that the process $(X_t,V_t)_{t\geq0}$ is obtained as a perturbation of Langevin dynamics by an overdamped Langevin diffusion in the position variable. This process is also known as the Hessian-free high-resolution dynamics \cite{LiZhaTao2022HFHR,WangWangZhu2026HFHR} in the sampling and machine learning community, where it was motivated by Nesterov's accelerated gradient method \cite{Nesterov1983Method}.\smallskip

In the following, we will assume that $\nu$ satisfies the Poincaré inequality
\begin{equation}\label{eq:PoincareAssu}
    \int_{\R^d} f^2\diff\nu \ \leq\ \frac{1}{m}\int_{\R^d} |\nabla f|^2\diff\nu\qquad\text{for all }f\in H^1(\nu)\cap L_0^2(\nu)\,.
\end{equation}
Lower bounds for the relaxation time of Langevin dynamics can be obtained by considering explicit test functions that yield an upper bound on the singular value gap. 

\begin{theorem}[Lower bounds]\label{thm:Langevinlower}
    \begin{enumerate}[(i)]
        \item If $\alpha = 0$, then
        \begin{equation*}
            \sing(\calL)\ \leq \ 2\gamma\,.
        \end{equation*}
        In particular, $\overline{t}_\rel \geq \frac{1-e^{-1}}{\gamma}$ and $t_\rel \geq \frac{1-e^{-1}}{2\gamma}$.

        \item In general, we have
        \begin{equation*}
            \sing(\calL)\ \leq\ (\gamma^{-1}+\alpha)\widetilde m\qquad\text{with}\qquad \widetilde m\ =\ \inf_{\substack{f\in L_0^2(\nu)\\ f\text{ affine}}}\frac{\norm{\calL_xf}}{\norm{f}}\,.
        \end{equation*}
        In particular, $\overline{t}_\rel\geq\frac{2(1-e^{-1})}{\widetilde m(\gamma^{-1}+\alpha)}$ and $t_\rel\geq\frac{1-e^{-1}}{\widetilde m(\gamma^{-1}+\alpha)}$.
    \end{enumerate}
\end{theorem}
The proof of \Cref{thm:Langevinlower} is given in \Cref{ssec:Langevinproof}.
Note that 
\begin{equation*}
    \widetilde m \ \geq \ \inf_{f\in L_0^2(\nu)\cap H^1(\nu)}\frac{\norm{\calL_xf}}{\norm{f}}\ =\ \inf_{f\in L_0^2(\nu)\cap H^1(\nu)}\frac{\langle f,-\calL_xf\rangle}{\norm{f}^2}  \ = \ m
\end{equation*}
by reversibility of $\calL_x$.
For quadratic potentials, affine functions attain the infimum on the right-hand side and $\widetilde m$ coincides with the inverse Poincaré constant $m$.\smallskip

Proving sharp upper bounds on the relaxation time is more challenging.
We have
\begin{equation*}
    \E_1^\calL(f,f)\ =\ -\int f(\alpha\calL_x+\gamma\calL_v)f\diff\mu\ \geq\ \min\left(\alpha m,\gamma\right)\int f^2\diff\mu
\end{equation*}
for all $f\in\dom(\calL)\cap L_0^2(\mu)$, so that the Poincaré constant of $\calL$ satisfies $\lambda(\calL)\geq \min(\alpha m,\gamma)$. In particular, for $\alpha,\gamma>0$, we immediately obtain
\begin{equation}\label{eq:Langevintrel}
    t_\rel\ \leq \ \max\left({\gamma}^{-1},\alpha^{-1} m^{-1}\right)\,.
\end{equation}
However, this bound does  not take into account the effect of the Hamiltonian flow, whose generator $\calH$ is antisymmetric in $L^2(\mu)$ and does therefore not contribute to the Poincaré constant $\lambda(\calL)$. In particular, the bound degenerates for $\alpha\to 0$, while the relaxation time stays finite. The two-component approach of \Cref{ssec:twoscale} yields the more refined bounds of \Cref{thm:Langevinupper}. 

Note that taking two independent copies $(X_t^i,V_t^i)_{t\geq 0}$, $i\in\{1,2\}$ of $(X_t,V_t)_{t\geq 0}$ again yields a perturbed Langevin dynamics on $\R^{2d}\times\R^{2d}$. Indeed, letting $\widetilde X = (X^1,X^2)$, $\widetilde V = (V^1,V^2)$, $\widetilde U(x_1,x_2) = U(x_1)+U(x_2)$, $\widetilde B = (B^1,B^2)$ and $\widetilde W = (W^1,W^2)$ yields
\begin{align*}
    \diff \widetilde X_t \ &=\ \widetilde V_t\diff t\, -\, \alpha \nabla \widetilde U(\widetilde X_t)\diff t \,+\, \sqrt{2\alpha}\diff \widetilde B_t\\
    \diff \widetilde V_t \ &=\ -\nabla \widetilde U(\widetilde X_t)\diff t \,-\, \gamma \widetilde V_t\diff t \,+\,\sqrt{2\gamma}\diff \widetilde W_t\,,
\end{align*}
so that the we can immediately consider the two-point motion.

\begin{theorem}[Upper bounds]\label{thm:Langevinupper}
    Assume that the potential $U$ satisfies the lower curvature bound $\nabla^2U(x) \geq  -c\cdot m I_d$ for all $x\in\R^d$ and some $c\in[0,\infty)$.
    \begin{enumerate}[(i)]
        \item If $\alpha = 0$, then
        \begin{equation*}
            t_\rel(P) \ \leq \ \frac{8e(\pi+2)}{1-e^{-1}}\max\left(\frac{2\gamma}{m},\frac{5+4c}{\gamma}\right)\,.
        \end{equation*}

        \item Suppose $\alpha>0$, and $\nabla^2U(x)\leq LI_d$ for a constant $L\in(0,\infty)$. Then
        \begin{align*}
            t_\rel(P)\ &\leq \ \frac{8e(\pi+2)}{1-e^{-1}}\max\left(\frac{2\gamma}{m+\alpha^2m^2},\frac{5+4c+2L\alpha^2}{\gamma}\right)\qquad\text{and}\\
            \gap(\calL)\ &\geq \ \frac{1-e^{-1}}{16e}\min\left(\frac{m+\alpha^2m^2}{2\gamma},\frac{\gamma}{5+4c+2L\alpha^2}\right)\,.
        \end{align*}
    \end{enumerate}
\end{theorem}

The proof of \Cref{thm:Langevinupper} is given in \Cref{ssec:Langevinproof} below.

\begin{remark}
    Up to improved constants, the upper bound in \Cref{thm:Langevinupper}(i) coincides with the upper bound on the relaxation time of Langevin dynamics obtained using space-time Poincaré inequalities under the same assumptions on $U$, see \cite[Theorem 1]{Cao2019Langevin} and \cite[Theorem 16]{EberleLoerler2024Divergence}. The proof presented here is simpler, as it avoids the divergence lemma. In case $\alpha = 0$, the bounds of \Cref{thm:Langevinlower,thm:Langevinupper} show that
    \begin{equation}\label{eq:Langevintrelbar}
        \frac{1-e^{-1}}{2}\max\left(\frac{1}{\gamma},\frac{2\gamma}{\widetilde m}\right)\ \leq \ t_\rel \ \leq \ 16e\max\left(\frac{5+4c}{\gamma},\frac{2\gamma}{m}\right)\,.
    \end{equation}
    In particular, in the convex case $c=0$, the upper and lower bounds are essentially of the same order for all choices of friction $\gamma\in (0,\infty)$. 
\end{remark}

\subsection{General diffusions}

Beyond the previous example of (perturbed) Langevin dynamics, let us consider more general  diffusion processes on $\R^d\times\R^d$ with product invariant probability measure $$\mu(\diff x\diff v)\ =\ \nu(\diff x)\kappa(\diff v)\,.$$ While we are not able to obtain a bound on the relaxation time in the fully general setting, we compute the first- and second-order collapses, thereby showing what is necessary in order to check the conditions of the two-component approach in \Cref{ssec:twoscale} once a more concrete model is given.

We again apply the two-component approach based on the projection
\begin{equation*}
    \Pi f(x,v) = \int f(x,w)\kappa(\diff w)\,,\qquad f\in L^2(\mu)\,,
\end{equation*}
and identify $\bbH^l = \Ran(\Pi)$ with $L^2_0(\R^d,\nu)$. We assume that the generator $\calL$ of the diffusion on a core $\D$ is given by
\begin{equation*}
    \calL g \ = \ \calL_xg + \calL_vg + b\cdot\nabla g\,,
\end{equation*}
where $\calL_x$ and $\calL_v$ are self-adjoint linear operators on $\bbH^l$ and $\bbH^h$, respectively, and the vector field $b\colon\R^{2d}\to\R^{2d}$ is divergence-free with respect to $\mu$, i.e.\ $\divg(be^{-U}) = 0$.
Then
\begin{equation*}
    \calL g = \calL_xg + b_x\cdot\nabla_x g\qquad\text{for all }g\in\D^l\,,\quad\text{where }b(x,v) = \begin{pmatrix}b_x(x,v)\\ b_v(x,v)\end{pmatrix}\,.
\end{equation*}
For $g\in\D^l$ we thus have
\begin{equation*}
    \E_1^\calL(g,g) = -\langle g,\calL_x g\rangle\,,
\end{equation*}
so that $\calC_1 = \calL_x$. Furthermore,
\begin{align*}
    \E_2^\calL(g,g) &= \norm{\calL g}^2 = \norm{\calL_xg}^2 + \norm{b_x\cdot\nabla g}^2 + 2\langle \calL_xg,b_x\cdot\nabla_x g\rangle\\
    &=\norm{\calL_xg}^2 + \langle\nabla_x g,T\nabla_x g\rangle + \langle g,[\calL_x,b_x\cdot\nabla_x]g\rangle\\
    &= \norm{\calL_xg}^2 + \langle\nabla_x g,T\nabla_x g\rangle + \langle g,[\calL_x,\overline b_x\cdot\nabla_x]g\rangle\,,
\end{align*}
where
\begin{equation*}
    T(x) = \int b_x(x,v)b_x(x,v)^\top\kappa(\diff v)\quad\text{and}\quad \overline b_x(x) = \int b_x(x,v)\kappa(\diff v)\,.
\end{equation*}
Therefore,
\begin{equation*}
    -{\calC}_2 = \calL_x^2 + \nabla_x^*T\nabla_x + [\calL_x,\overline{b}_x\cdot\nabla_x]\,.
\end{equation*}
In the example of perturbed Langevin dynamics in \Cref{ssec:Langevin} above, we have $T=I_d$ and $\overline{b}_x = 0$.

Assumption \eqref{eq:A1} is equivalent to
\begin{equation*}
    -\int f(\calL_xf+\calL_vf)\diff\hat\mu\ \geq\ \gamma\norm{f-\Pi f}^2\qquad\text{for all }f\in\D_0\,,
\end{equation*}
This is satisfied if $\calL_v$ satisfies a Poincaré inequality on $\bbH^h$.
Assumption \eqref{eq:B2} says
\begin{equation*}
    -\int g{\calC}_2g\diff\mu\ \geq\ m_2\int g^2\diff\mu\qquad\text{for all }f\in\D_0^l\,,
\end{equation*}
It can be verified easily in case $[\calL_x,\overline{b}_x\cdot\nabla_x] = 0$, while otherwise it is more intricate.
Finally, for \eqref{eq:C2}, set $\calA = \calL-\gamma(\Pi-I)$, so that it remains to verify
\begin{equation*}
    \norm{(I-\Pi)\calA^*\calA g} \leq C_2\norm{{\calC}_2g}\qquad\text{for all }g\in\D_0^l\,.
\end{equation*}
A computation or $g\in\D_0^l$ shows that
\begin{align*}
    \calL^*\calL g 
    &= \calL_x^2g - b_x\cdot\nabla_x\calL_xg + \calL_x(b_x\cdot\nabla_xg) + \calL_v(b_x\cdot\nabla_xg) - b\cdot\nabla(b_x\cdot\nabla_xg)\\
    &= \calL_x^2g + [\calL_x,b_x\cdot\nabla_x]g + \calL_v(b_x\cdot\nabla_xg) - b\cdot\nabla(b_x\cdot\nabla_xg)\,,\\
    (I-\Pi)\calA^*\calA g &= (I-\Pi)(\calL^*-\gamma(\Pi-I))\calL g =(I-\Pi)\calL^*\calL g + \gamma(b_x-\overline{b}_x)\cdot\nabla_xg\,.
\end{align*}

\begin{example}

    If $\calL_x = 0$, then $-{\calC}_2 = \nabla_x^*T\nabla_x$. Therefore,
    for all $g\in\bbH^l$,
    \begin{align*}
        \norm{{\calC}_2g}^2 &= \int(\nabla_x^*T\nabla_xg)^2\diff\mu = \int (T\nabla g)\cdot\nabla\nabla^*T\nabla g\diff\mu\\
        &= \int (T\nabla g)\cdot(\nabla^*\nabla+\nabla^2U)(T\nabla g)\diff\mu \\
        &= \int\norm{\nabla (T\nabla g)}_F^2\diff\mu + \int\nabla g\cdot (T\nabla^2UT)\nabla g\diff\mu\,,\qquad\text{and}
    \end{align*}
    \begin{equation*}
        \calL^*\calL g = \calL_v(b_x\cdot\nabla_xg)-b\cdot\nabla(b_x\cdot\nabla_xg) = \calL_v(b_x\cdot\nabla_xg) - (b_x\cdot\nabla_x)^2g - b_v\cdot\nabla_vb_x \nabla_xg\,.
    \end{equation*}
\end{example}

\section{Proofs for main results}\label{sec:Proofs1}

\subsection{Proofs for Section \ref{ssec:gapbounds}}\label{ssec:gapboundsproofs}

The identity
\begin{equation}\label{eq:keyidentity}
    \calL\overline P_tf \ = \ \frac{1}{t}(P_tf-f)\qquad\text{for all }f\in L_0^2(\mu)\text{ and }t>0
\end{equation}
is crucial for proofs of upper bounds on $\overline{t}_\rel$.
In order to prove the upper bound in \Cref{thm:trel_svgap}, we prove the following slightly stronger statement based on \eqref{eq:keyidentity}.

\begin{theorem}\label{thm:BCbound}
    Suppose that there exist constants $B,C\in[0,\infty)$ such that 
    \begin{equation}\label{eq:BCbound}
        \norm{\overline P_tf}\ \leq\ C\norm{\calL\overline P_tf} + \frac{B}{t}\norm{f}\qquad\text{for all }f\in L_0^2(\mu)\text{ and }t>0\,.
    \end{equation}
    Then 
    \begin{equation*}
        \norm{\overline P_t f}\ \leq \ \frac{2C+B}{t}\norm{f}\qquad\text{for all }f\in L_0^2(\mu)\text{ and }t>0\,.
    \end{equation*}
    In particular, $\overline t_\rel(P)\leq (2C+B)e$.
\end{theorem}
\begin{proof}
    This follows from the key identity \eqref{eq:keyidentity}
    which together with \eqref{eq:BCbound} yields
    \begin{align*}
        \norm{\overline P_sf}\leq \frac{2C}{s}\norm{f} + \frac{B}{s}\norm{f}\qquad\text{for all }s>0\,,
    \end{align*}
    so that $\norm{\overline P_s}_{L_0^2(\mu)\to L_0^2(\mu)}\leq\frac{1}{e}$ for all $s\geq (2C+B)e$.
\end{proof}

\begin{proof}[Proof of \Cref{thm:trel_svgap}]
    The upper bound follows from \Cref{thm:BCbound} with $B=0$ and $C = \frac{1}{\sing(\calL)}$. For the lower bound, note that for any $f\in L_0^2(\mu)\cap\dom(\calL)$ we have $P_tf-f = \int_0^t P_s\calL f\diff s$, so that
    \begin{equation*}
        \norm{P_tf-f} \leq \int_0^t \norm{P_s\calL f}\diff s\leq t\norm{\calL f}\,.
    \end{equation*}
    This yields
    \begin{equation*}
        \norm{P_tf}\geq \norm{f} - \norm{P_tf-f}\geq \norm{f} - t\norm{\calL f}
    \end{equation*}
    For $\sigma>\sing(\calL)$, we can choose $f\in L_0^2(\mu)\cap\dom(\calL)$ with $\norm{\calL f}<\sigma\norm{f}$, so that
    \begin{equation*}
        \norm{P_tf}> (1-\sigma t)\norm{f}\,,
    \end{equation*}
    and in particular $t_\rel(P)\geq (1-e^{-1})/\sigma$, so that taking $\sigma\downarrow\sing(\calL)$ yields the lower bound for $t_\rel(P)$.
    For the lower bound on $\overline t_\rel(P)$, we similarly have 
    \begin{equation*}
        \norm{\overline P_tf-f} \leq \frac{1}{t}\int_0^t\norm{P_sf-f}\diff s \leq \frac{1}{t}\int_0^ts\norm{\calL f}\diff s = \frac{t}{2}\norm{\calL f}\,,
    \end{equation*}
    and can conclude analogously.
\end{proof}

\subsection{Proofs for Section \ref{ssec:twopointmotion}}\label{ssec:twopointproofs}

\begin{proof}[Proof of \Cref{thm:L_L2}]
    Suppose that $\lambda = a+ib\in\spec(-\calL|_{L_0^2(\mu)\cap\dom(\calL)})$. We first assume that $\lambda$ is not in the residual spectrum. Then for every $\varepsilon>0$ there exists $g\in L^2_\C(S,\mu)$ such that $g\in\dom_\C(\calL)$ and $\norm{\calL g-\lambda g}_{L^2_\C(\mu)}<\varepsilon\norm{g}_{L^2_\C(\mu)}$. Here the subscript $\C$ denotes the complexification. Since $\calL$ maps real-valued function to real-valued functions, $\overline{\calL g} = \calL\overline g$ for all $g\in\dom_\C(\calL)$. Therefore, $\overline{\calL g-\lambda g} = \calL\overline g-\overline\lambda\overline g$, and thus 
    \begin{equation*}
        \norm{\calL\overline g - \overline\lambda\overline g}_{L^2_\C(\mu)} = \norm{\overline{\calL g - \lambda g}}_{L^2_\C(\mu)} = \norm{\calL g - \lambda g}_{L^2_\C(\mu)}<\varepsilon\norm{g}_{L^2_\C(\mu)} = \varepsilon\norm{\overline g}_{L^2_\C(\mu)}\,.
    \end{equation*}
    Now let $f = g\otimes\overline{g}$, i.e.\ $f(x,y) = g(x)\overline{g(y)}$. Then $f\in \dom(\calL^{(2)})$ and $\calL^{(2)}f = (\calL g)\otimes\overline{g} + g\otimes(\calL\overline{g})$, so that
    \begin{equation*}
        \calL^{(2)}f - 2af = \calL^{(2)}f - \lambda f - \overline{\lambda}f = (\calL g-\lambda g)\otimes\overline{g} + g\otimes(\calL\overline g-\overline{\lambda}\overline{g})\,.
    \end{equation*}
    In particular, 
    \begin{equation*}
        \norm{\calL^{(2)}f-2af}_{L^2_\C(\mu\otimes\mu)} \leq \norm{\calL g-\lambda g}_{L^2_\C(\mu)}\norm{\overline{g}}_{L^2_\C(\mu)} + \norm{g}_{L^2_\C(\mu)}\norm{\calL\overline{g}-\overline{\lambda}\overline{g}}_{L^2_\C(\mu)}\leq 2\varepsilon\norm{f}_{L^2_\C(\mu\otimes\mu)}\,.
    \end{equation*}
    We conclude that 
    \begin{equation*}
        \norm{\calL^{(2)}f}_{L^2_\C(\mu\otimes\mu)}\leq 2(|a|+\varepsilon)\norm{f}_{L^2_\C(\mu\otimes\mu)}\,,
    \end{equation*}
    and hence $\sing(\calL^{(2)})\leq 2\left\lvert\Re(\lambda)\right\rvert$.

    On the other hand, if $\lambda$ is in the residual spectrum, then $\overline{\Ran(\calL-\lambda I)}\neq L^2_\C(\mu)$, i.e.\ there exists $g\in L^2_\C(\mu)$ such that 
    \begin{equation*}
        \langle \calL f-\lambda f,g\rangle_{L^2_\C(\mu)} = 0\qquad\text{for all }f\in L_{0,\C}^2(\mu)\cap\dom_\C(\calL)\,.
    \end{equation*}
    Thus $g\in\dom_\C((\calL-\lambda I)^*)$ and 
    \begin{equation*}
        (\calL-\lambda I)^*g = \calL^*g-\overline{\lambda}g = 0\,,
    \end{equation*}
    i.e.\ $\overline{\lambda}$ is an eigenvalue of $\calL^*$. Then $f = g\otimes\overline{g}$ satisfies
    \begin{align*}
        \calL^{(2)}f &= (\calL g)\otimes\overline{g} + g\otimes(\calL\overline{g})= \overline{\lambda}g\otimes\overline{g} + \lambda g\otimes\overline{g} = 2\Re(\lambda)f\,,
    \end{align*}
    and thus again $\sing(\calL^{(2)})\leq 2\left\lvert\Re(\lambda)\right\rvert$. 

    Since this holds for all $\lambda\in \spec(-\calL|_{L_0^2(\mu)\cap\dom(\calL)})$, we conclude that
    \begin{equation*}
        \sing(\calL^{(2)})\leq 2\inf\{\left\lvert\Re(\lambda)\right\rvert:\lambda\in\spec(-\calL|_{L_0^2(\mu)\cap\dom(\calL)})\} = 2\gap(\calL)
    \end{equation*}
    as claimed.
\end{proof}

\begin{proof}[Proof of \eqref{eq:psibound}]
    Let $\lambda\in\R$ and $f\in\dom(\calL)\cap L_0^2(\mu)$ with $\norm{f}=1$. Then setting $g = f\otimes \overline{f}$ yields $g\in\dom(\calL^{(2)})
    \cap L_0^2(\mu\otimes\mu)$ with $\norm{g}=1$. Denoting $r = (\calL-i\lambda)f$, we have
    \begin{align*}
        \calL^{(2)}g &= (\calL f)\otimes \overline{f} + f\otimes(\calL\overline{f}) = (r+i\lambda f)\otimes\overline{f} + f\otimes(\overline{r}-i\lambda\overline{f})\
        =\ r\otimes \overline{f} + f\otimes\overline{r}\,,
    \end{align*}
    so that
    \begin{equation*}
        \norm{\calL^{(2)}g}_{L^2(\mu\otimes\mu)}\leq \norm{r}_{L^2(\mu)}\norm{\overline{f}}_{L^2(\mu)} + \norm{f}_{L^2(\mu)}\norm{\overline{r}}_{L^2(\mu)} = 2\norm{r}_{L^2(\mu)} = 2\norm{(\calL-i\lambda)f}_{L^2(\mu)}\,.
    \end{equation*}
    In particular, since $\lambda$ and $f$ were arbitrary, we obtain $\sing(\calL^{(2)})\leq 2\Psi(\calL)$ by definition of the singular value gap.
\end{proof}
\begin{proof}[Proof of \Cref{cor:twopoint}]
    Wei's theorem and \eqref{eq:psibound} yield \eqref{eq:twopointupper}, from which the upper bounds follow directly.
    
    For the lower bound, by decomposing $L^2(\mu) = \langle 1\rangle\oplus L_0^2(\mu)$, we observe that
    \begin{equation*}
        L^2(\mu\otimes\mu)\ \cong\ \bigl(\langle 1\rangle\otimes \langle 1\rangle\bigr)\oplus \bigl(L_0^2(\mu)\otimes \langle 1\rangle\bigr) \oplus \bigl(\langle 1\rangle\otimes L_0^2(\mu)\bigr)\oplus \bigl(L_0^2(\mu)\otimes L_0^2(\mu)\bigr)\,,
    \end{equation*}
    and therefore
    \begin{eqnarray*}
        L^2_0(\mu\otimes\mu)\ \cong\ \bigl(L_0^2(\mu)\otimes \langle 1\rangle\bigr) \oplus \bigl(\langle 1\rangle\otimes L_0^2(\mu)\bigr)\oplus \bigl(L_0^2(\mu)\otimes L_0^2(\mu)\bigr)\,.
    \end{eqnarray*}
    The restriction of $P_t^{(2)}$ to these subspaces is $P_t\otimes I$, $I\otimes P_t$, and $P_t\otimes P_t$, respectively. In particular, since $P_t$ is a contraction,
    \begin{equation*}
        \norm{P_t^{(2)}}_{L_0^2(\mu\otimes\mu)\to L_0^2(\mu\otimes\mu)} \ = \ \max\bigl(\norm{P_t}_{L_0^2(\mu)\to L_0^2(\mu)},\norm{P_t}_{L_0^2(\mu)\to L_0^2(\mu)}^2 \bigr) \ =\  \norm{P_t}_{L_0^2(\mu)\to L_0^2(\mu)}\,.
    \end{equation*}
    By \Cref{thm:trel_svgap}, this yields $t_\rel(P) = t_\rel(P^{(2)}) \geq \frac{1-e^{-1}}{\sing(\calL^{(2)})}$, concluding the proof.
\end{proof}

\subsection{Proofs for Section \ref{ssec:twoscale}}

\begin{proof}[Proof of \Cref{thm:liftedPoincare}]
    Let $f\in\dom(\E)$ with $\Pi f\in\dom(\E)$. Then, since $-{\calC}$ is bounded below, $\Ran({\calC}) = \bbH^l$ and there exists $g\in\dom({\calC})$ such that $\Pi f = -{\calC}g$. Therefore,
    \begin{align*}
        \norm{\Pi f}^2 &= \langle\Pi f,\Pi f\rangle = \langle\Pi f,-{\calS}g\rangle = \E(\Pi f,g)\\
        &=\E(\Pi f-f,g) + \E(f,g)\,.
    \end{align*}
    Since $\Pi f-f\in\bbH^h$, \eqref{eq:C} yields $\E(\Pi f-f,g)\leq C\norm{\Pi f-f}\norm{{\calC}g}$, while the second term can be bounded as
    \begin{equation*}
        \E(f,g) \leq \E(f,f)^{1/2}\E(g,g)^{1/2}\leq\frac{1}{\sqrt{m}}\E(f,f)^{1/2}\norm{{\calC}g}
    \end{equation*}
    by \eqref{eq:B}. Since $-{\calC}g = \Pi f$, this yields
    \begin{equation*}
        \norm{\Pi f}^2 \leq \left(C\norm{\Pi f-f} + \frac{1}{\sqrt{m}}\E(f,f)^{1/2}\right)\norm{\Pi f}\,.
    \end{equation*}
    Together with the orthogonal decomposition $\norm{f}^2 = \norm{\Pi f}^2 + \norm{f-\Pi f}^2$ this yields the claim.
\end{proof}

\begin{proof}[Proof of \Cref{cor:LPI}]
    \Cref{thm:liftedPoincare} applied to $\E_i^\calL$, $i\in\{1,2\}$, and using \eqref{eq:A1} and \eqref{eq:A2}, respectively, yields
    \begin{align*}
        \norm{f}^2 \leq \frac{2}{m_i}\E_i^\calL(f,f) + (1+2C_i^2)\norm{f-\Pi f}^2\leq \bigg(\frac{2}{m} + \frac{(1+2C_i)^2}{\gamma^i}\bigg)\E_i^\calL(f,f)\,.
    \end{align*}
    The bounds for the relaxation time of $P$ and its time averages then follow from \eqref{eq:trelpoincare} and \Cref{thm:trel_svgap}.
\end{proof}

We now turn to the proof of \Cref{thm:FPI}. Recall that ${\calA} = \calL-\gamma(\Pi-I)$. The identity \eqref{eq:keyidentity} shows that for any $t>0$ and $f\in L_0^2(\mu)$, we have
\begin{equation}\label{eq:keyid_A}
    {\calA}\overline{P}_tf \ = \ \gamma(I-\Pi)\overline{P}_tf\ +\ \frac{1}{t}(P_t f-f)\,.
\end{equation}
This identity is crucial since it shows that for large $t$, ${\calA}\overline P_tf \approx \gamma(I-\Pi)\overline P_t$, allowing to shift between the ${\calA}$ and $\gamma(\Pi-I)$ parts of the generator, typically corresponding to transport and dissipation. In particular, \eqref{eq:keyid_A} implies that
\begin{equation}
    \Pi {\calA}\overline P_tf = \frac{1}{t}(P_tf-f)\to 0\quad\text{ as }t\to\infty\,.
\end{equation}
A first consequence of \eqref{eq:keyid_A} is the following.
\begin{lemma}\label{lem:crossterm}
    For all $t > 0$ and $f \in L^2_0(\mu)$,
    \begin{equation}
        \norm{{\calA}\overline P_t f}^2 + \gamma^2 \norm{\Pi\overline P_tf-\overline P_tf}^2 \leq \norm{\calL\overline P_t f}^2 + \frac{2\gamma}{t}\langle f - P_t f, \Pi\overline P_t f\rangle\,.
    \end{equation}
\end{lemma}
\begin{proof}
    Since 
    \begin{equation*}
        \langle {\calA} P_t f, \Pi P_t f - P_t f \rangle \leq 0
    \end{equation*}
    and $\Pi$ is an orthogonal projection, we have
    \begin{align*}
        \MoveEqLeft\langle {\calA}\overline P_t f, \Pi \overline P_t f - \overline P_t f \rangle \geq \langle {\calA}\overline P_t f, \Pi\overline P_t f \rangle = \langle \Pi {\calA}\overline P_t f, \Pi\overline P_t f \rangle \\
        &= \frac{1}{t} \langle \Pi(P_tf - f), \Pi\overline P_t f \rangle = \frac{1}{t} \langle P_t f - f, \Pi\overline P_t f \rangle\,.
    \end{align*}
    Hence,
    \begin{align*}
        \norm{\calL\overline P_t f}^2
        &= \norm{{\calA}\overline P_t f}^2 
        + \gamma^2 \norm{\Pi\overline P_tf-\overline P_tf}^2 
        + 2\gamma \langle {\calA}\overline P_t f, \Pi\overline P_tf-\overline P_tf \rangle \\
        &\geq \norm{{\calA}\overline P_t f}^2 
        + \gamma^2 \norm{\Pi\overline P_t f -\overline P_t f}^2 
        + \frac{2\gamma}{t} \langle P_t f - f, \Pi\overline P_t f \rangle.
    \end{align*}
\end{proof}

By \Cref{thm:BCbound}, in order to obtain an upper bound on $\overline t_\rel$---or, equivalently, a lower bound on the singular value gap---it therefore suffices to bound $\norm{\overline P_tf}^2$ by $\norm{{\calA}\overline P_tf}^2+\gamma^2\norm{\Pi\overline P_tf -\overline P_tf}^2$.

\begin{proof}[Proof of \Cref{thm:FPI}]
    \Cref{thm:liftedPoincare} applied to $\E_2^{\calA}$ shows that
    \begin{equation*}
        \norm{\overline P_tf}^2\ \leq\ \frac{2}{m_2}\norm{{\calA}\overline P_tf}^2\ +\ (1+2C_2^2)\norm{\overline P_tf - \Pi \overline P_tf}^2
    \end{equation*}
    for all $f\in\dom(\E_2^{\calA})$ with $\Pi f\in\dom(\E_2^{\calA})$, so that \Cref{lem:crossterm} yields
    \begin{align*}
        \norm{\overline P_tf}^2\ \leq\  \max\biggl(\frac{2}{m_2},\frac{1+2C_2^2}{\gamma^2}\biggr)\Bigl(\norm{\calL\overline P_tf}^2 + \frac{2\gamma}{t}\norm{f-P_tf}\norm{\overline P_tf}\Bigr)\,.
    \end{align*}
    Rearranging thus gives
    \begin{equation*}
        \norm{\overline P_tf} \ \leq \ \max\biggl(\frac{2}{m_2},\frac{1+2C_2^2}{\gamma^2}\biggr)^{1/2}\norm{\calL\overline P_tf} \ + \ \frac{4\gamma}{t}\max\biggl(\frac{2}{m_2},\frac{1+2C_2^2}{\gamma^2}\biggr)\norm{f-P_tf}\,.
    \end{equation*}
    By \Cref{thm:BCbound} with $C = \max\bigl(\frac{2}{m_2},\frac{1+2C_2^2}{\gamma^2}\bigr)^{1/2}$ and $B = 8\gamma C^2$, we obtain
    \begin{align*}
        \overline t_\rel(P) \ & \leq \ (2C+8\gamma C^2)e\ \leq \ 4e\max\biggl(\sqrt{\frac{2}{m_2}},\frac{\sqrt{1+2C_2^2}}{\gamma},\frac{8\gamma}{m_2},\frac{4(1+2C_2^2)}{\gamma}\biggr)\\
        \ &= \ 16e \max\biggl(\frac{2\gamma}{m_2},\frac{1+2C_2^2}{\gamma}\biggr)\,.
    \end{align*}
    The lower bound for the singular value gap follows by \Cref{thm:trel_svgap}.
\end{proof}

\section{Proofs for examples}\label{sec:Proofs2}

\subsection{Bounds for lifted random walks}\label{ssec:ProofLRW}

Recall that the lifted random walk introduced in \Cref{sec:liftedRW} on $\Z_n^d$ is the continuous-time Markov chain on $S = \Z_n^d\times V$, where $V = \{\pm e_i:i\in\{1,\dots,d\}\}$, with generator
\begin{equation*}
    \calL f(x,v) \ = \ f(x+v,v)-f(x,v) \ + \ \gamma (\Pi f-f)(x,v)\,,
\end{equation*}
where 
\begin{equation*}
    \Pi f(x,v)\ =\ \frac{1}{2d}\sum_{w\in V}f(x,w)\,.
\end{equation*}
Its invariant probability measure is $\mu = \unif(S) = \nu\otimes\kappa$ with $\nu = \unif(\Z_n^d)$ and $\kappa=\unif(V)$.

We begin by proving an upper bound on the singular value gap by identifying an explicit test function that relaxes slowly.

\begin{lemma}\label{lem:lowerboundLRW}   There exists a universal constant $C\in (0,\infty )$ such that for all $d,n\in\mathbb N$ with $n\ge 2$ and all $\gamma\in (0,\infty )$, 
     \begin{equation}
      s(\calL )\ \le \   C\cdot d^{-1}\cdot\min\Big({n^{-2}\max (1,\gamma^{-1} )},{\gamma}\Big)\  \,.
    \end{equation}
\end{lemma}

\begin{proof}
    We identify an explicit function $f\colon\mathbb Z_n^d\times V\to\mathbb C$ such that $\int f\diff\mu =0$ and $\|\calL f\|/\| f\|$ is small. To this end we make the ansatz
    $$f(x,v)\ =\ \ell (x_1)1_{v\neq -e_1}+k(x_1)1_{v\neq e_1}+h(x_1)1_{v\not\in\{ e_1,-e_1\} }\, $$
    with functions $\ell ,k,h\colon\mathbb Z_n\to\mathbb C$.
    Let $\ell'(x):=\ell (x+1)-\ell (x)$ and $k'_-(x)=k'(x-1)$ denote the right and left-sided discrete derivatives. Then
    \begin{eqnarray*}
        (\calL f)(x,v) &=& \ell'(x_1)1_{v=e_1}+\gamma \ell (x_1)(1_{v=-e_1}-\tfrac{1}{2d})-k'_-(x_1)1_{v=-e_1}+\gamma k(x_1)(1_{v=e_1}-\tfrac{1}{2d})\\
        &&+\ \gamma h(x_1)(1_{v=e_1}+1_{v=-e_1}-\tfrac{1}{d})\\
        &=&1_{v=e_1}(\ell'(x_1)+\gamma k(x_1)+\gamma h(x_1))+1_{v=-e_1}(-k_-'(x_1)+\gamma\ell(x_1)+\gamma h(x_1))\\
        &&-\ \tfrac{\gamma}{d}(\tfrac 12\ell (x_1)+\tfrac 12 k(x_1)+h(x_1)) \,.
    \end{eqnarray*}
    We will choose $k,\ell ,h$ such that the first two terms vanish, i.e.,
    \begin{equation}\label{eq:klh}
        -\ell'-\gamma k\ =\ \gamma h\ =\ k'_- -\gamma\ell\, ,\qquad \gamma\ell-\ell'\ =\ \gamma k+k'_-\, .
    \end{equation}
    If this condition is satisfied then we also have
    \begin{equation}\label{eq:hkl}
        h\ =\ -\tfrac 12 (k+\ell )+\tfrac{1}{2\gamma}(k'_--\ell')\, , 
    \end{equation}
    and thus
    $$(\calL f)(x,v)\ =\ \tfrac{1}{2d}(k'_--\ell')(x_1)\, .$$
    To make $\calL f$ as small as possible we choose 
    $$k(x)\ =\ \alpha\, e^{2\pi ix/n}\, ,\quad \ell (x)\ =\ \beta \, e^{2\pi ix/n}$$
    with $\alpha ,\beta\in \mathbb C$. Then
    \begin{eqnarray*}
        (\gamma k+k'_-)(x)& =& (\gamma +1-e^{-2\pi i/n})\alpha e^{2\pi ix/n}\, ,\\
        (\gamma \ell+\ell')(x)& =& (\gamma +1-e^{2\pi i/n})\beta e^{2\pi ix/n} ,
    \end{eqnarray*}
    and thus condition \eqref{eq:klh} is satisfied if we set
    $$\alpha \ =\ \gamma +1-e^{2\pi i/n}\, ,\qquad \beta \ =\ \gamma +1-e^{-2\pi i/n}\, .$$
    For this choice we obtain
    \begin{eqnarray}\label{eq:kplp}
        (k'_--\ell')(x) 
        &=& (2+\gamma )(1-e^{-2\pi i/n})(1-e^{2\pi i/n})e^{2\pi ix/n}\, .   
    \end{eqnarray}
    In particular,
    \begin{eqnarray}\label{eq:normLf}
     (\calL f)(x,v) &=& \tfrac{2+\gamma}{2d}(1-e^{-2\pi i/n})(1-e^{2\pi i/n})e^{2\pi ix/n}\, ,\\
        \|\calL f\|& =& \tfrac{2+\gamma}{2d}\left|(1-e^{-2\pi i/n})(1-e^{2\pi i/n})\right|\ \in\ \Theta \left(\tfrac{1+\gamma}{dn^2}\right)\, .
    \end{eqnarray}
    This has to be compared to $\| f\|$. By \eqref{eq:hkl},
    \begin{eqnarray*}
        f(x,v) &=& \left( \ell (x_1)+k(x_1)+h(x_1)\right) 1_{v\not\in\{ e_1,-e_1\} }+\ell (x_1)1_{v=e_1}+k(x_1)1_{v=-e_1}\\
        &=&\tfrac 12\left( \ell (x_1)+k(x_1)+\tfrac{1}{\gamma} (k'_--\ell')(x_1)\right)1_{v\not\in\{ e_1,-e_1\} }+\ell (x_1)1_{v=e_1}+k(x_1)1_{v=-e_1}\, .
    \end{eqnarray*}
    Therefore, we have
    $$\| f\|^2\ \ge\ \tfrac{d-1}{2d}\int \left(\ell +k+\tfrac{1}{\gamma}(k'_--\ell')\right)^2\diff \mathrm{Unif}(\mathbb Z_n) \, .$$
    By \eqref{eq:kplp} and since
    $$(\ell +k)(x)\ =\ (\alpha +\beta )e^{2\pi ix/n}\ =\ \left(2\gamma +(1-e^{-2\pi i/n})(1-e^{2\pi i/n})\right) e^{2\pi ix/n}\, ,$$
    we conclude that
    $\| f\|\ \in \Theta \left( \gamma +\tfrac{1}{\gamma n^2}\right)$, and thus
    $$s(\calL )\ \le\ \frac{\|\calL f\| }{\| f\| }\ \in\  O\left(\min \left(\frac{1+\gamma}{\gamma dn^2} ,
    \frac{(1+\gamma )\gamma n^2}{dn^2}\right)\right)\ =\ O\Big( d^{-1}\min\Big({n^{-2}\max (1,\gamma^{-1} )},{\gamma}\Big) \Big)\, .$$
\end{proof}

In order to prove upper bounds on the relaxation time of the lifted random walk, we employ the two-component approach of \Cref{ssec:twoscale}. As explained in \Cref{ex:kinetic}, we set $\bbH^l = \Ran(\Pi)$ and identify functions in $\bbH^l$ with functions $g\colon\Z_n^d\to\R$. We let $\calA$ be such that
\begin{equation*}
    \calL\ = \ \calA + \gamma(\Pi-I)\,.
\end{equation*}
For $v\in V$ and $f\colon\Z_n^d\to\R$ denote $\partial_vf(x) = f(x+v)-f(x)$. Then the adjoint of $\partial_v$ in $L^2(\nu)$ is $\partial_v^*f(x) = f(x)-f(x-v)$, so that
\begin{align*}
    \partial_v^*\partial_vf(x) &= f(x+v)-2f(x)+f(x-v)\,,\quad\text{and}\\
    \partial_v^2f(x) &= f(x+2v)-2f(x+v)+f(x)\,.
\end{align*}
Below, we will eventually apply the two-component approach to the associated two-point motion in order to obtain bounds on the relaxation time $t_\rel$. It is however instructive to at first consider the lifted random walk itself.

\begin{lemma}\label{lem:upperboundLRW}
    The lifted random walk on $\Z_n^d$ satisfies
    \begin{equation}
        \overline t_\rel \ \leq \ 16e\max\Big(\frac{dn^2\gamma}{20},\frac{1+4d}{\gamma}\Big)\qquad\text{and}\qquad t_\rel\ \leq \ \frac{dn^2}{10}+\frac{1+4d}{\gamma} \,.
    \end{equation}
    In particular,
    \begin{equation*}
        \overline{t}_\rel \ \leq \ 16e\max\Big(\frac{dn^2\min(1,\gamma)}{20},\frac{1+4d}{\gamma}\Big)\,.
    \end{equation*}
\end{lemma}
\begin{proof}
    We apply \Cref{thm:FPI} for the upper bound on $\overline{t}_\rel$ and \Cref{cor:LPI}(i) for the upper bound on $t_\rel$. Clearly \eqref{eq:A1} is satisfied.
    For $f\circ\pi\in\bbH^l$ we have
    \begin{equation*}
        \E_i^\calL(f\circ\pi,f\circ\pi) = \frac{i}{2d}\langle f\circ\pi,(\Delta f)\circ\pi\rangle \geq \frac{i}{d}\big(1-\cos(2\pi/n)\big)\norm{f\circ\pi}^2\,,
    \end{equation*}
    so that \eqref{eq:B1} and \eqref{eq:B2} are satisfied with $m_i = \frac{i}{d}(1-\cos(2\pi/n))\in\Theta(d^{-1}n^{-2})$, since the spectrum of $\Delta$ is 
    \begin{equation*}
        \spec(\Delta) = \left\{2\sum_{i=1}^d\Big(1-\cos\Big(\frac{2\pi k_i}{n}\Big)\Big):k_i\in\{0,\dots,n-1\}\right\}\,.
    \end{equation*}
    It remains to verify \eqref{eq:C2}. To this end, it suffices to prove \eqref{eq:C2'}, i.e.\ that
    \begin{equation*}
        \norm{(I-\Pi){\calA}^*{\calA}(g\circ\pi)}\ \leq \ C_2\norm{{\calC}_2(g\circ\pi)}\qquad\text{for all }g\colon\Z_n^d\to\R\,.
    \end{equation*}
    Since for $g\colon\Z_n^d\to\R$ we have $\calA(g\circ\pi)(x,v) = \partial_vf(x)$, we obtain
    \begin{align*}
        \norm{{\calA}^*{\calA}(g\circ\pi)}^2  &= \frac{1}{2d}\sum_{v\in V}\int (\partial_v^*\partial_vg(x))^2\nu(\diff x) = \frac{1}{2d}\sum_{v\in V}\int (\partial_v^2g(x))^2\nu(\diff x)\\
        &\leq\frac{1}{2d}\sum_{v,w\in V}\int(\partial_v\partial_wg(x))^2\nu(\diff x) = \frac{1}{2d}\sum_{v,w\in V}\langle\partial_v\partial_wg,\partial_v\partial_wg\rangle_{L^2(\nu)}\\
        &=\frac{1}{2d}\sum_{v,w\in V}\langle\partial_v^*\partial_vg,\partial_w^*\partial_wg\rangle_{L^2(\nu)} = \frac{1}{2d}\int \Big(\sum_{v\in V}\partial_v^*\partial_vg(x)\Big)^2\nu(\diff x)\\
        &=\frac{1}{2d}\int(2\Delta g(x))^2\nu(\diff x) = 2d\norm{{\calC}_2(g\circ\pi)}^2\,,
    \end{align*}
    so that \eqref{eq:C2} is satisfied with $C_2^2 = 2d$. Analogously, since $({\calA}^*+{\calA})(g\circ\pi) = {\calA}^*{\calA}(g\circ\pi)$ and ${\calC}_1 = \frac{1}{2}{\calC}_2$, \eqref{eq:C1} is also satisfied with $C_1^2 = 2d$.

    Therefore, \Cref{cor:LPI} and \Cref{thm:FPI} yield
    \begin{align*}
        t_\rel\ &\leq\ \frac{2}{m_1}+\frac{1+2C_1^2}{\gamma}\ \leq\ \frac{dn^2}{10}+\frac{1+4d}{\gamma}\,,\qquad\text{and}\\
        \overline t_\rel\ &\leq\ 16e\max\Big(\frac{2\gamma}{m_2},\frac{1+2C_2^2}{\gamma}\Big)\ \leq\ 16e\max\Big(\frac{dn^2\gamma}{20},\frac{1+4d}{\gamma}\Big)\,,
    \end{align*}
    respectively, where we used that $m_i\geq\frac{20\cdot i}{dn^2}$.
    The combined bound follows from \eqref{eq:bartrel_trel}.
\end{proof}

In order to obtain a sharper upper bound on the relaxation time $t_\rel$ of the non-averaged semigroup, we consider the two-point motion.

\begin{lemma}\label{lem:LRWtwopoint}
    The two-point motion of the lifted random walk satisfies
    \begin{equation*}
        \sing(\calL^{(2)}) \ \geq \ \frac{1-e^{-1}}{8e}\min\left(\frac{20}{dn^2\min(1,\gamma)},\frac{\gamma}{9+16d}\right)\,.
    \end{equation*}
\end{lemma}
\begin{proof}
    For the two-point motion of the lifted random walk, the state space is given by $$S = \Z_n^d\times\Z_n^d\times V\quad\text{with}\quad V = \{(\pm e_i,\pm e_j): i,j=1,\dots,d\}\,,$$ and the invariant measure is $\mu = \nu\otimes\unif(V)$ with $\nu = \unif(\Z_n^d\times\Z_n^d)$. We let $\pi(x,y,v,w) = (x,y)$.
    For $(v,w)\in V$ and $f\colon\Z_n^d\times\Z_n^d\to\R$ we use the notation
    \begin{equation*}
        \partial^x_vf(x,y) = f(x+v,y)-f(x,y)\quad\text{and}\quad \partial^y_wf(x,y) = f(x,y+w)-f(x,y)\,.
    \end{equation*}
    These operators extend naturally to $f\colon S\to\R$ by acting on the first components. Furthermore, for $f\colon S\to\R$ we let
    \begin{align*}
        \Pi_v f(x,y,v,w) &= \frac{1}{2d}\sum_{i = 1}^d\bigl(f(x,y,e_i,w)+f(x,y,-e_i,w)\bigr)\quad\text{and}\\ \Pi_w f(x,y,v,w) &= \frac{1}{2d}\sum_{i = 1}^d\bigl(f(x,y,v,e_i)+f(x,y,v,-e_i)\bigr)\,,
    \end{align*}
    as well as $\Pi f = \Pi_v\Pi_wf$.
    The generator of the two-point motion is then given by
    \begin{equation*}
        \calL = \calT + \gamma(\Pi_v-I)+\gamma(\Pi_w-I) = \calA+\gamma(\Pi-I)
    \end{equation*}
    with 
    \begin{align*}
        \mathcal{T} = \partial^x_v + \partial^y_w\qquad\text{and}\qquad \calA = \calT - \gamma(\Pi_v-I)(\Pi_w-I)\,.
    \end{align*}
    We employ the two-component approach of \Cref{ssec:twoscale}. Since $\calA$ is non-positive, \eqref{eq:A1} is clearly satisfied.
    Furthermore, as above, since $(\partial^x_v)^* = \partial_{-v}^x$ and $(\partial^y_w)^* = \partial^y_{-w}$, the first-order collapse is given by
    \begin{align*}
        \calC_1(g\circ\pi) = \frac{1}{2}\Pi(\calL+\calL^*)(g\circ\pi) = \frac{1}{2d}\bigl(\Delta_xg+\Delta_yg\bigr)\,,
    \end{align*}
    where $\Delta_xg = \sum_{i=1}^dg(x+e_i,y)-2g(x,y)+g(x-e_i,y)$, and $\Delta_y$ is defined analogously by acting on the second component. To determine the second-order collapse, we see that for $g\in\bbH^l$, i.e.\ functions $g\colon\Z_n^d\times\Z_n^d\to\R$ we have
    \begin{align*}
        \E_2^\calL(g,g) &= \norm{\calA g}^2 =  \frac{1}{(2d)^2}\sum_{(v,w)\in V}\int \calA g(x,y,v,w)^2\mu(\diff x\diff y)\,.
    \end{align*}
    We have
    \begin{align*}
        \MoveEqLeft\sum_{(v,w)\in V}\calA g(x,y,v,w)^2 = \sum_{i,j=1}^d\bigl(\calA g(x,y,e_i,e_j)^2+\calA g(x,y,-e_i,e_j)^2\\
        &\qquad\qquad\qquad\qquad\qquad + \calA g(x,y,e_i,-e_j)^2+\calA g(x,y,-e_i,-e_j)^2\bigr)\\
        &=\sum_{i,j=1}^d\Bigl(\bigl(\partial^x_{e_i}g+\partial^y_{e_j}g\bigr)^2 + \bigl(\partial^x_{-e_i}g+\partial^y_{e_j}g\bigr)^2 +\bigl(\partial^x_{e_i}g+\partial^y_{-e_j}g\bigr)^2+\bigl(\partial^x_{-e_i}g+\partial^y_{-e_j}g\bigr)^2\Bigr)\\
        &=2\sum_{i,j=1}^d(\partial^x_{e_i}g)^2 + (\partial^x_{-e_i}g)^2 + (\partial^y_{e_j}g)^2 + (\partial^y_{-e_j}g)^2 + (\partial^x_{e_i}g + \partial^x_{-e_i}g)(\partial^y_{e_j}g + \partial^y_{-e_j}g)\\
        &=2\Delta_xg\Delta_yg+2\sum_{i,j=1}^d\Bigl((\partial^x_{e_i}g)^2 + (\partial^x_{-e_i}g)^2 + (\partial^y_{e_j}g)^2 + (\partial^y_{-e_j}g)^2\Bigr)
    \end{align*}
    Let us denote
    \begin{equation*}
        \nabla_x g(x,y) = \bigl(\partial^x_{e_i}g(x,y)\bigr)_{i=1}^d\quad\text{and}\quad \nabla_y g(x,y) = \bigl(\partial^y_{e_i}g(x,y)\bigr)_{i=1}^d\,.
    \end{equation*}
    Then
    \begin{align*}
        \E_2^\calL(g,g) &= \frac{1}{d}\int\bigl(|\nabla_xg|^2 + |\nabla_yg|^2\bigr)\diff\mu + \frac{1}{2d^2}\int \Delta_xg\Delta_yg \diff\mu\\
        &=-\int g\left(\frac1d\Delta_xg+\frac{1}{d}\Delta_yg - \frac{1}{2d^2}\Delta_x\Delta_yg\right)\diff\mu\,,
    \end{align*}
    so that the second-order collapse is given by
    \begin{equation*}
        \calC_2 = \frac{1}{d}\Delta_x + \frac{1}{d}\Delta_y - \frac{1}{2d^2}\Delta_x\Delta_y\,.
    \end{equation*}
    In particular, for $g\in\bbH^l$, we obtain
    \begin{equation*}
        \E_2^\calL(g,g) = \langle g,-\calC_2g\rangle\geq 2\langle g,-\calC_1 g\rangle = \frac{1}{d}\langle g,-\Delta_xg-\Delta_yg\rangle \geq \frac{2}{d}(1-\cos(2\pi/n))\norm{g}^2\,,
    \end{equation*}
    so that \eqref{eq:B2} is satisfied with $m_2 = \frac{2}{d}(1-\cos(2\pi/n))\in\Theta(d^{-1}n^{-2})$, and \eqref{eq:B1} is satisfied with $m_1 = m_2/2$. It remains to verify \eqref{eq:C2'}. To this end, note that $$\calA^* = \calT^* - \gamma(\Pi_v-I)(\Pi_w-I)\,,$$ and that for $g\in\bbH^l$ we have
    \begin{align*}
        (\Pi_v-I)(\Pi_w-I)\calA g= (\Pi_v-I)(\Pi_w-I)(\partial^x_vg(x,y)+\partial^y_wg(x,y))=0\,.
    \end{align*}
    Hence $\calA^*\calA g = \calT^*\calT g$, so that
    \begin{align*}
        \bigl(\calA^*\calA g(x,y,v,w)\bigr)^2  &= \bigl((\partial^x_{v})^*\partial^x_vg + (\partial^x_{v})^*\partial^y_wg + (\partial^y_{w})^*\partial^x_vg + (\partial^y_{w})^*\partial^y_wg\bigr)^2(x,y)\\
        &\leq 4\bigl(((\partial^x_{v})^*\partial^x_vg)^2 + ((\partial^y_{w})^*\partial^y_wg)^2 + ((\partial^x_{v})^*\partial^y_wg)^2 + ((\partial^y_{w})^*\partial^x_vg)^2\bigr)(x,y)\,.
    \end{align*}
    As in the proof of \Cref{lem:upperboundLRW} above, integrating with respect to $\mu$, the first two terms can be bounded as
    \begin{align*}
        \MoveEqLeft\frac{1}{(2d)^2} \sum_{v,w\in V}\int\bigl(((\partial^x_{v})^*\partial^x_vg)^2 + ((\partial^y_{w})^*\partial^y_wg)^2\bigr)\diff\nu \leq \frac{1}{2d}\int \bigl((2\Delta_xg)^2+(2\Delta_yg)^2\bigr)\diff\nu\\
        &= \frac{2}{d}\norm{(\Delta_x+\Delta_y)g}^2 \leq 2d\norm{\calC_2 g}^2\,.
    \end{align*}
    For the remaining term, we have
    \begin{align*}
        \MoveEqLeft\frac{1}{(2d)^2} \sum_{v,w\in V}\int\bigl(((\partial^x_{v})^*\partial^y_wg)^2 + ((\partial^y_{w})^*\partial^x_vg)^2\bigr)\diff\nu\\
        &=\frac{1}{(2d)^2} \sum_{v,w\in V}\left( \bigl\langle(\partial^x_v)^*\partial^y_wg,(\partial^x_v)^*\partial^y_wg\bigr\rangle_{L^2(\nu)} + \bigl\langle(\partial^y_w)^*\partial^x_vg,(\partial^y_w)^*\partial^x_vg\bigr\rangle_{L^2(\nu)}\right)\\
        &=\frac{2}{(2d)^2} \sum_{v,w\in V}\bigl\langle(\partial^x_v)^*\partial^x_vg,(\partial^y_w)^*\partial^y_wg\bigr\rangle_{L^2(\nu)}\\
        &= \frac{2}{d^2}\bigl\langle\Delta_xg,\Delta_yg\bigr\rangle_{L^2(\nu)}\leq \norm{\calC_2g}^2\,,
    \end{align*}
    since 
    \begin{equation*}
        \norm{\calC_2g}^2 = \Bigl\lVert\frac{1}{d}\Delta_xg + \frac{1}{d}\Delta_yg - \frac{1}{2d^2}\Delta_x\Delta_yg\Bigr\rVert^2\geq \frac{2}{d^2}\langle\Delta_xg,\Delta_yg\rangle_{L^2(\nu)}\,.
    \end{equation*}
    Overall, this yields
    \begin{equation*}
        \norm{\calA^*\calA g}^2\leq (8d+4)\norm{\calC_2g}^2\,,
    \end{equation*}
    so that \eqref{eq:C2'} holds with $C_2^2 = 8d+4$.
    This yields
    \begin{equation*}
        \overline{t}_\rel(P^{(2)})\ \leq\ 16e \max\left(\frac{\gamma dn^2}{20},\frac{9+16d}{\gamma}\right)\
    \end{equation*}
    by \Cref{thm:FPI}, where we used that $m_2\geq\frac{40}{dn^2}$.

    The first-order part and \eqref{eq:C1} can be treated analogously to \Cref{lem:upperboundLRW} above, so that we obtain
    \begin{equation*}
        \overline{t}_\rel(P^{(2)})\ \leq\ 16e \max\left(\frac{\min(1,\gamma) dn^2}{20},\frac{9+16d}{\gamma}\right)\,.
    \end{equation*}

    The lower bound on the singular value gap of the two-point motion then follows from the relation $\sing(\calL^{(2)})\geq\frac{2(1-e^{-1})}{\overline{t}_\rel(P^{(2)})}$, see \Cref{thm:trel_svgap}.
\end{proof}

\begin{proof}[Proof of \Cref{thm:LRW}]
    Applying \Cref{cor:twopoint} and \Cref{lem:LRWtwopoint} yields
    \begin{equation*}
        t_\rel(P)\ \leq\ \frac{\pi+2}{\sing(\calL^{(2)})}\ \leq\ \frac{8e(\pi+2)}{1-e^{-1}}\max\left(\frac{dn^2\min(1,\gamma)}{20},\frac{9+16d}{\gamma}\right)\,.
    \end{equation*}
    Similarly, the lower bound on the spectral gap follows from \Cref{thm:L_L2}.
    
    Finally, by \eqref{eq:bartrel_trel}, \Cref{thm:trel_svgap} and \Cref{lem:lowerboundLRW}, we obtain the matching lower bound 
    \begin{equation*}
         t_\rel(P)\ \geq \ \frac{e-1}{e^2}\overline t_\rel(P)\ \geq \ \frac{2(1-e^{-1})^2e^{-1}}{\sing(\calL)} \ \geq \ c\cdot d\cdot \max\left(n^2\min(1,\gamma),\gamma^{-1}\right)
    \end{equation*}
    with a constant $c\in(0,\infty)$.
    \end{proof}

\subsection{Bounds for switching flows perturbed by noise}\label{ssec:Proofsflows}

\begin{proof}[Proof of \Cref{thm:shear}]
    Conditions \eqref{eq:A1} and \eqref{eq:A2} are satisfied, and
    \eqref{eq:B1} and \eqref{eq:B2} hold with
    $m_1$ and $m_2$ given by \eqref{eq:m1m2flows}. Furthermore, since
    $\mathcal E_1^{\calA}(f,g)=-\langle f,\calL_xg\rangle$, \eqref{eq:C1} holds 
    by \eqref{eq:S1flows} with
    $C_1=1$, proving the first bound on $t_\rel$ by \Cref{cor:LPI}(i).
    
    Thus it only remains to verify \eqref{eq:C2}. By \eqref{eq:S2flows},
    \begin{equation}\label{eq:lbS2flows}
        \norm{\calC_2g}^2\ \ge\ \left\| \mathcal L_x^2g\right\|^2+\left\| T\mathbin{:} \nabla^2g \right\|^2\, .
    \end{equation}
    Here we have used that by integration by parts, 
    $$-\langle \calL_x^2g, T\mathbin{:} \nabla^2g\rangle\ =\ \langle \calL_x^2g, \nabla^*T \nabla g\rangle\ =\ \int (\nabla\calL_xg)\cdot T\nabla\calL_xg\diff\nu\ \ge\ 0.$$
    Moreover, since $T$ is symmetric, $T\mathbin{:}\nabla^2g=\nabla\cdot T\nabla g=-\nabla^*(T\nabla g)=-(T^{1/2}\nabla )^*T^{1/2}\nabla g$. Therefore, integration by parts shows that
    \begin{equation}\label{Eq:Bochnerflows}
        \| T\mathbin{:}\nabla^2g\|^2\ =\ \int \|T^{1/2}\nabla T^{1/2}\nabla g\|_F^2\diff\nu\ =\ \int \|T^{1/2}(\nabla^2g) T^{1/2}\|_F^2\diff\nu\, .
    \end{equation}
    The last step holds since for $B\in\mathbb R^{d\times d}$ and $i,j\in\{ 1, \ldots ,d\}$,
    we habe $(B\nabla B\nabla g)_{ij}=\sum_{k,l}B_{ik}\partial_kB_{jl}\partial_lg=(B(\nabla^2g)B^T)_{ij}$.
    By \eqref{eq:AstarAflows}, the definition of $C_{Q,T}$ and \eqref{Eq:Bochnerflows}, we obtain
    \begin{eqnarray*}
    \label{eq:AstarAflowsproj}
        \| (I-\Pi )\calA^*\calA g\|^2 & =& \| (\beta (v)\otimes \beta (v)-T)\mathbin{:}\nabla^2g\|^2\\
        &=&\| (\beta (v)\otimes \beta (v))\mathbin{:}\nabla^2g\|^2-\| T\mathbin{:}\nabla^2g\|^2\\
        &=& \int \nabla^2g\mathbin{:} (Q-T\otimes T)\mathbin{:}\nabla^2 g\diff\nu\\
        &\le& C_{Q,T}^2\int \left\| T^{1/2}(\nabla^2g)T^{1/2}\right\|_F^2\diff\nu\ =\ 
        C_{Q,T}^2\| T\mathbin{:}\nabla^2g\|^2.
    \end{eqnarray*}
    In combination with \eqref{eq:lbS2flows}, this shows that \eqref{eq:C2'} and \eqref{eq:C2} are satisfied with
    $C_2=C_{Q,T}$.

    To obtain bounds on the non-averaged relaxation time and the spectral gap, we consider the two-point motion $(X^1_t,V^1_t,X^2_t,V^2_t)$ on $S\times S$, i.e., the components $(X^1_t,V^1_t)$ and $(X^2_t,V^2_t)$ are independent copies of the Markov process 
    $(X_t,V_t)$. Let $\overline X_t:=(X_t^1,X_t^2)$ and $\overline V_t:=(V_t^1,V_t^2)$. Note that the two-point motion is a Markov process of a similar form 
    as $(X_t,V_t)$ with the following differences:
    The diffusion matrix and the drift for the two-point motion are 
    $$\overline A=\begin{pmatrix}
        A&0\\0&A
    \end{pmatrix}\, ,\quad \overline\beta (v)=\begin{pmatrix}
        \beta (v^1)\\ \beta (v^2)
    \end{pmatrix}\, ,$$
    and the generator for $(\overline V_t)_{t\ge 0}$ is
    $\gamma (\Pi_1-I)+\gamma (\Pi_2-I)$ instead of $\gamma (\Pi-I)$,
    where $(\Pi_1f)(x^1,v^1,x^2,v^2)=\int f(x^1,w,x^2,v^2)\kappa (\diff w)$, and
    $\Pi_2$ is defined correspondingly. In components, $\overline A_{ij}=A\cdot\delta_{ij}$ and $\overline{\beta}_i(v)=\beta (v_i)$.
    The second and fourth moments of $\overline\beta $ are
    \begin{eqnarray*}
        \overline{T}_{ij} &=& \int \overline{\beta}_i\otimes \overline{\beta}_j\diff(\kappa\otimes\kappa )\ =\ T\cdot\delta_{ij}\, ,\\ 
        \overline{Q}_{ijkl} &=&\int \overline{\beta}_i\otimes \overline{\beta}_j\otimes \overline{\beta}_k\otimes \overline{\beta}_l\diff(\kappa\otimes\kappa\otimes\kappa\otimes\kappa )\ =\ \left\{
    \begin{array}{cc}
        Q & \text{if }i=j=k=l, \\
        T\otimes T & \text{if }i=j\neq k=l, \\
         T\otimes T & \text{if }i=k\neq j=l, \\
          T\otimes T & \text{if }i=l\neq j=k, \\
          0& \text{otherwise.}
    \end{array}\right.
    \end{eqnarray*}
    In particular, $(\overline{Q}-\overline{T}\otimes\overline{T})_{ijkl}$ equals $Q-T\otimes T$ for $i=j=k=l$, $T\otimes T$ for $i=k\neq j=l$ or $i=l\neq j=k$, and $0$ in all other cases.
    Condition \eqref{eq:A1} is still satisfied since $$I-\Pi_1+I-\Pi_2\ =\ (I-\Pi )+(I-\Pi_1)(I-\Pi_2)\, .$$
    Moreover, the second-order collapse of the two-point motion takes a similar form as that of the original process, i.e., \eqref{eq:S2flows} holds with $A$ and $T$ replaced by $\overline{A}$ and $\overline{T}$. Therefore, \eqref{eq:B2} is satisfied with $m_2$ given by \eqref{eq:m1m2flows}, i.e., the same constant as for the original process. Finally, Condition \eqref{eq:C2} can be verified similarly as for the original process by noting that by integration by parts,
    \begin{eqnarray*}
      \int \left( \overline{T}\mathbin{:}\nabla^2g\right)^2\diff(\nu\otimes\nu )& =& \int \left( {T}\mathbin{:}\nabla^2_{11}g+T\mathbin{:}\nabla^2_{22}g\right)^2\diff(\nu\otimes\nu )\\
      & =& \int \left\{\left( {T}\mathbin{:}\nabla^2_{11}g\right)^2+ \left({T}\mathbin{:}\nabla^2_{22}g\right)^2+2\left(  {T}\mathbin{:}\nabla^2_{12}g\right)^2\right\}\diff(\nu\otimes\nu )\, ,
    \end{eqnarray*}
    and on the other hand,
    \begin{eqnarray*}
      \lefteqn{\int \nabla^2g \mathbin{:} (\overline{Q}-\overline{T}\otimes\overline{T})\mathbin{:}\nabla^2g\diff(\nu\otimes\nu )}\\
      & =& \int \nabla^2_{11}g \mathbin{:} (\overline{Q}-\overline{T}\otimes\overline{T})\mathbin{:}\nabla^2_{11}g\diff(\nu\otimes\nu )+\int \nabla^2_{22}g \mathbin{:} (\overline{Q}-\overline{T}\otimes\overline{T})\mathbin{:}\nabla^2_{22}g\diff(\nu\otimes\nu )\\
      &&{}+4\int \nabla^2_{12}g \mathbin{:} (\overline{T}\otimes\overline{T})\mathbin{:}\nabla^2_{12}g\diff(\nu\otimes\nu )
      \, .
    \end{eqnarray*}
    Therefore, similarly as above, \eqref{eq:C2} is satisfied with $C_2=\max (C_{Q,T},\sqrt 2)$, and we conclude by \Cref{cor:FPI}.
\end{proof}

\subsection{Bounds for Langevin dynamics}\label{ssec:Langevinproof}

In this section, we give proofs for the upper and lower bounds on the relaxation time for the perturbed Langevin dynamics $(X_t,V_t)_{t\geq 0}$ given by
\begin{align*}
    \diff X_t \ &=\ V_t\diff t\, -\, \alpha \nabla U(X_t)\diff t \,+\, \sqrt{2\alpha}\diff B_t^{(1)}\\
    \diff V_t \ &=\ -\nabla U(X_t)\diff t \,-\, \gamma V_t\diff t \,+\,\sqrt{2\gamma}\diff B_t^{(2)}
\end{align*}
that are stated in \Cref{ssec:Langevin}.

\begin{proof}[Proof of \Cref{thm:Langevinlower}]
    We show upper bounds on the singular value gap of $\calL$ by considering explicit test functions.

    \begin{enumerate}[(i)]
        \item In case $\alpha = 0$, for small $\gamma\in(0,\infty)$, the dynamics is dominated by the Hamiltonian flow. Since this preserves the Hamiltonian $H$, we choose 
        \begin{equation*}
            f(x,v) = H(x,v)-\int H\diff\mu = H(x,v)-\int U\diff\nu-\frac{d}{2}\,.
        \end{equation*}
        Then, on the one hand,
        \begin{align*}
            \calL f(x,v) = \gamma(d-|v|^2)\,,\qquad\text{so that}\qquad \int (\calL f)^2\diff\mu = 2d\gamma^2\,. 
        \end{align*}
        On the other hand,
        \begin{align*}
            \int f^2\diff\mu &= \int H^2\diff\mu - \left(\int H\diff\mu\right)^2\\
            &=\int U^2\diff\nu+\frac{1}{4}\int|v|^2\,\kappa(\diff v) + d\int U\diff\nu -\left(\int U\diff\nu\right)^2 - d\int U\diff\nu - \frac{d^2}{4}\\
            &\geq \frac{1}{4}\left(\int |v|^4\,\kappa(\diff v) - d^2\right) = \frac{d}{2}\,.
        \end{align*}
        Therefore,
        \begin{equation*}
            \sing(\calL)\ \leq\ \frac{\norm{\calL f}}{\norm{f}}\ \leq\ 2\gamma\,.
        \end{equation*}

        \item For $f(x,v) = (x-\gamma^{-1}v)\cdot e - \int (x-\gamma^{-1}v)\cdot e\,\mu(\diff x\diff v)$ with $e\in\mathbb{S}^{d-1}$ we obtain
        \begin{equation*}
            \calL f(x,v) = -(\gamma^{-1}+\alpha)e\cdot\nabla U(x)
        \end{equation*}
        and 
        \begin{align*}
            \norm{f}_{L^2(\mu)}^2 &= \int \left(x\cdot e-\int (x\cdot e)\nu(\diff x)\right)^2\,\nu(\diff x) + \gamma^{-2}\int (v\cdot e)^2\,\kappa(\diff v)\\
            &\geq \int \left(x\cdot e-\int (x\cdot e)\nu(\diff x)\right)^2\,\nu(\diff x)\,.
        \end{align*}
        Therefore,
        \begin{equation*}
            \sing(\calL)\leq \frac{\norm{\calL f}}{\norm{f}} \leq (\gamma^{-1}+\alpha)
            \left(\frac{\int (\partial_eU(x))^2\,\nu(\diff x)}{\int \left(x\cdot e-\int (x\cdot e)\nu(\diff x)\right)^2\,\nu(\diff x)}\right)^{1/2},
        \end{equation*}
        and taking the infimum over $e\in\mathbb{S}^{d-1}$ yields the claim since $\calL_x f = \partial_eU$.

    \end{enumerate}

\end{proof}

In order to apply the two-component approach of \Cref{ssec:twoscale}, we set 
\begin{equation*}
    \Pi f(x,v) = \int f(x,w)\kappa(\diff w)\,,\qquad f\in L^2(\mu)\,,
\end{equation*}
and identify $\bbH^l$ with $L^2(\R^d,\nu)$.
The adjoint of the generator $\calL = \calH+\alpha\calL_x+\gamma\calL_v$ in $L^2(\mu)$ is $\calL^* = -\calH + \alpha\calL_x + \gamma\calL_v$. 
The symmetric first- and second-order collapse of $\calL$ are thus given by
\begin{equation}\label{eq:Langevincollapse}
    \begin{aligned}
        \calC_1g \ &= \ \frac{1}{2}\Pi(\calL^*+\calL)g\ =\ \alpha \calL_xg\,,\\
        \calC_2g \ &= \ -\Pi\calL^*\calL g\ =\ -\alpha^2\calL_x^2g + \calL_xg
    \end{aligned}
\end{equation}
for $g\in\dom(\calC_i)\subseteq L^2(\nu)$.

\begin{proof}[Proof of \Cref{thm:Langevinupper}]

Condition \eqref{eq:A1} is satisfied since $\calL_v$ has spectral gap $1$ on $L^2(\kappa)$.
Furthermore, by \eqref{eq:PoincareAssu} and \eqref{eq:Langevincollapse}, conditions \eqref{eq:B1} and \eqref{eq:B2} are satisfied with
\begin{equation*}
    m_1 = \alpha m\qquad\text{and}\qquad m_2 = \alpha^2m^2 + m\,.
\end{equation*}
It remains to verify condition \eqref{eq:C2}. At first, note that
\begin{equation*}
    \norm{\calC_2g}^2 \geq \alpha^4\norm{\calL_x^2g}^2 + \norm{\calL_xg}^2\,. 
\end{equation*}
Now set $\calA = \calL-\gamma(\Pi-I)$. Then $\calA g = \calL g = \alpha\calL_xg+v\cdot\nabla_xg$ for all $g\in L^2(\nu)\cap\dom(\calL)$, so that
\begin{equation*}
    (\calL_v-\gamma(\Pi-I))^*\calA g = (\calL_v - \gamma(\Pi-I))\calA g = 0
\end{equation*}
We therefore obtain
\begin{align*}
    \calA^*\calA g &= (\alpha\calL_x+\calH)^*\calA g = (\alpha\calL_x-\calH)\calA g\\
    &=(\alpha\calL_x-v\cdot\nabla_x+\nabla U(x)\cdot\nabla_v)(\alpha\calL_xg+v\cdot\nabla_xg)\\
    &=\alpha^2\calL_x^2 - (v\otimes v)\mathbin{:}\nabla_x^2g + \nabla U(x)\cdot\nabla_xg + \alpha v\cdot\nabla^2U(x)\cdot\nabla_xg\,,
\end{align*}
and thus 
\begin{align}\label{eq:AstarALangevin}
    (I-\Pi)\calA^*\calA g(x,v) = (I-v\otimes v)\mathbin{:}\nabla_x^2g(x) + \alpha v\cdot\nabla^2U(x)\nabla_x g(x)\,.
\end{align}
Since $\kappa = \mathcal{N}(0,I_d)$, we have
\begin{equation*}
    \int\left((I-v\otimes v)\mathbin{:}\nabla_x^2g\right)^2\kappa(\diff v) = 2\norm{\nabla_x^2g}^2_F\,,
\end{equation*}
where $\norm{\cdot}_F$ denotes the Frobenius norm. 
Bochner's identity yields
\begin{align*}
    \int \norm{\nabla_x^2g}_F^2\diff\nu &= \int (\calL_xg)^2\diff\nu - \int\nabla_xg\cdot\nabla^2U\nabla_xg\diff\nu\\
    &\leq \norm{\calL_xg}_{L^2(\nu)}^2 + cm\int|\nabla_xg|^2\diff\nu \leq (1+c)\norm{\calL_xg}_{L^2(\nu)}^2\,.
\end{align*}
Again using Bochner's identity and the upper bound $\nabla^2U(x)\leq L\cdot I_d$, the second term in \eqref{eq:AstarALangevin} can be bounded as
\begin{align*}
    \norm{\alpha v\cdot\nabla^2U\nabla_xg}_{L^2(\mu)}^2 &= \alpha^2\int|\nabla^2U\nabla_xg|^2\diff\nu \leq L\alpha^2\int \nabla_xg\cdot\nabla^2U\nabla_xg\diff\nu\\
    &=L\alpha^2\left(\int (\calL_xg)^2\diff\nu - \int\norm{\nabla_x^2g}_F^2\diff\nu\right) \leq L\alpha^2\norm{\calL_xg}_{L^2(\nu)}^2\,.
\end{align*}
Overall, since the two terms are orthogonal in $L^2(\mu)$, we thus obtain
\begin{equation*}
    \norm{(I-\Pi)\calA^*\calA g}_{L^2(\mu)}^2\ \leq \ \bigl(2(1+c)+L\alpha^2\bigr)\norm{\calL_xg}_{L^2(\nu)}^2\ \leq\ \bigl(2(1+c)+L\alpha^2\bigr)\norm{\calC_2g}^2\,,
\end{equation*}
so that \eqref{eq:C2} holds with $C_2^2 = 2(1+c)+L\alpha^2$. We thus conclude by \Cref{cor:FPI}.

\end{proof}

\appendix

\section{Appendix}\label{appendix}

For sake of completeness, we include the short proofs of some auxiliary or well-known results.

\begin{proof}[Proof of \Cref{eq:bartrel_trel}]
    For any $n\in\N$ and $s\geq 0$ we have $\norm{P_{nt_\rel(P)+s}}_{L_0^2(\mu)\to L_0^2(\mu)}\leq e^{-n}$. This yields
    \begin{equation*}
        \norm{\overline P_t}_{L_0^2(\mu)\to L_0^2(\mu)}\leq \frac{t_\rel(P)}{t}\sum_{n=0}^\infty e^{-n} = \frac{t_\rel(P)}{t(1-e^{-1})}\,,
    \end{equation*}
    so that $\norm{\overline P_t}_{L_0^2(\mu)\to L_0^2(\mu)}\leq \frac{1}{e}$ for all $t\geq \frac{e}{1-e^{-1}}t_\rel(P)$ as claimed.
\end{proof}

\begin{proof}[Proof of \Cref{lem:gaprelations}]
    If $\lambda(\calL) = 0$, the statement is trivial. If $\lambda(\calL)>0$, the transition semigroup satisfies $\norm{P_t}_{L_0^2(\mu)\to L_0^2(\mu)}\leq\exp(-\lambda(\calL)t)$, see \cite[Theorem 4.2.5]{BGL2014Analysis}. Therefore, for any $\alpha\in\C$ with $\Re(\alpha)<\lambda(\calL)$ and $f\in L_0^2(\mu)$,
    \begin{align*}
        \int_0^\infty |e^{\alpha t}|\norm{P_t f}\diff t \leq C(\alpha)\norm{f}\,,
    \end{align*}
    where $C(\alpha)$ is a finite constant depending only on $\alpha$. Thus the resolvent $(-\alpha I-\calL)^{-1}f = \int_0^\infty e^{\alpha t}P_tf\diff t$ exists in $L^2(\mu)$ and defines a bounded linear operator, so that $\alpha$ is not contained in the spectrum of $(-\calL)|_{L_0^2(\mu)\cap\dom(\calL)}$ and $\gap(\calL)\geq\lambda(\calL)$. Similarly, for all $f\in L_0^2(\mu)$ we have
    \begin{equation*}
        \int_0^\infty \norm{P_tf}\diff t \leq \int_0^\infty\exp(-\lambda(\calL)t)\diff t\norm{f} = \frac{1}{\lambda(\calL)}\norm{f}\,,
    \end{equation*}
    so that the potential operator $(-\calL)^{-1}f = \int_0^\infty P_tf\diff t$ exists in $L^2(\mu)$. Its operator norm $\norm{(-\calL)^{-1}}_{L_0^2(\mu)\to L_0^2(\mu)}$ coincides with the inverse singular value gap and is bounded above by $\frac{1}{\lambda(\calL)}$, proving the second relation.
\end{proof}

For the convenience of the reader, we include a proof of \Cref{thm:wei}. It is a slight reformulation of the proof of \cite[Theorem 1.3]{Wei2021Diffusion}.

\begin{proof}[Proof of \Cref{thm:wei}]
    Let $f\in\dom(H)$ and $t\geq 0$ such that $t\Psi\geq \frac{\pi}{2}$, where we let $\Psi = \Psi(H)$. It suffices to show that
    \begin{equation*}
        \norm{Q_tf}\ \leq\ e^{-t\Psi + \frac{\pi}{2}}\norm{f}\,.
    \end{equation*}
    We set $\tau = \frac{\pi}{4\Psi}$. For a function $\chi\in C^1([0,t+2\tau])$ with $\chi(0) = \chi(t+2\tau)=0$, set $f_1(s) = \chi(s)Q_sf$ and $f_2(s) = \chi'(s)Q_sf$. Then
    \begin{equation*}
        f_2 = \partial_tf_1 - Hf_1\,.
    \end{equation*}
    Taking the Fourier transform 
    \begin{equation*}
        \hat f_j(\lambda) = \int_0^{t+2\tau}f_j(s)e^{-i\lambda s}\diff s\,,\qquad j\in\{1,2\},\quad\lambda\in\R\,,
    \end{equation*}
    yields
    \begin{equation*}
        \hat f_2(\lambda) = (i\lambda -H)\hat f_1(\lambda)\qquad\text{and thus}\qquad \Psi \leq \frac{\norm{\hat f_2}}{\norm{\hat f_1}}
    \end{equation*}
    by definition of $\Psi$.
    Plancherel's identity yields $\norm{f_j} = \sqrt{2\pi}\norm{\hat f_j}$, so that $\norm{f_2}\geq \Psi\norm{f_1}$. Plugging in the definitions of $f_1$ and $f_2$ yields
    \begin{equation*}
        0 \leq \int_0^{t+2\tau}\bigl(\chi'(s)^2-\Psi^2\chi(s)^2\bigr)\norm{Q_sf}^2\diff s\,.
    \end{equation*}
    We are thus led to choose $\chi$ such that the function $h = (\chi')^2-\Psi^2\chi^2$ has a small positive value at $0$, a large negative value at $t$, and is suitably small in between. Choose
    \begin{equation*}
        \chi(s) = \begin{cases}
            \sin(\Psi s)&0\leq s\leq \tau,\\
            \frac{1}{\sqrt{2}}e^{\Psi(s-\tau)}&\tau\leq s\leq t-\tau,\\
            e^{\Psi(t-2\tau)}\sin(\Psi(t+2\tau-s))& t-\tau\leq s\leq t+2\tau\,.
        \end{cases}
    \end{equation*}
    Then $\chi\in C^1([0,t+2\tau])$ with $\chi(0) = \chi(t+2\tau)=0$ and
    \begin{equation*}
        h(s) = \begin{cases}
            \Psi^2\cos(2\Psi s)&0\leq s\leq \tau,\\
            0&\tau\leq s\leq t-\tau,\\
            \Psi^2e^{2\Psi( t-2\tau)}\cos(2\Psi(t+2\tau-s))& t-\tau\leq s\leq t+2\tau\,,
        \end{cases}
    \end{equation*}
    see \Cref{fig:plot_chi,fig:plot_h}. Using that $s\mapsto \norm{Q_sf}^2$ is non-increasing, we obtain
    \begin{align*}
        0&\leq \int_0^{t+2\tau}h(s)\norm{Q_sf}^2\diff s = \int_0^{\tau}h(s)\norm{Q_sf}^2\diff s+\int_{t-\tau}^{t+2\tau}h(s)\norm{Q_sf}^2\diff s\\
        &\leq \int_0^{\tau}h(s)\norm{Q_0f}^2\diff s+\int_{t-\tau}^{t}h(s)\norm{Q_tf}^2\diff s + \int_{t}^{t+2\tau}h(s)\norm{Q_{t+\tau}f}^2\diff s\\
        &=\frac{\Psi}{2}\norm{f}^2 - \frac{\Psi}{2}e^{2\Psi(t-2\tau)}\norm{Q_tf}^2 + 0\,,
    \end{align*}
    and rearranging yields the claim.
        
    \begin{figure}[ht]
        \centering
        \begin{tikzpicture}[scale=0.8]
            \begin{axis}[
                width=12cm,
                height=7cm,
                domain=0:4.5708,
                samples=300,
                xmin=0,
                xmax=4.5708,
                ymin=0,
                ymax=6,
                axis lines=none,
                xtick=\empty,
                ytick=\empty,
                xlabel={},
                ylabel={},
                clip=false,
            ]
            
            \draw[thick]
                (axis cs:0,0) -- (axis cs:3+pi/2,0);
            
            \addplot[
                domain=0:pi/4,
                smooth,
                thick,
                blue
            ]
            {sin(deg(x))};
            
            \addplot[
                domain=pi/4:3-pi/4,
                smooth,
                thick,
                purple
            ]
            {exp(x-pi/4)/sqrt(2)};
            
            \addplot[
                domain=3-pi/4:3+pi/2,
                smooth,
                thick,
                green!50!black
            ]
            {exp(3-pi/2)*sin(deg(3+pi/2-x))};
            
            \addplot[
                dashed,
                thin,
                gray
            ]
            coordinates {(0,0) (0,5)};
            
            \addplot[
                dashed,
                thin,
                gray
            ]
            coordinates {(pi/4,0) (pi/4,5)};
            
            \addplot[
                dashed,
                thin,
                gray
            ]
            coordinates {(3-pi/4,0) (3-pi/4,5)};
            
            \addplot[
                dashed,
                thin,
                gray
            ]
            coordinates {(3,0) (3,5)};
            
            \addplot[
                dashed,
                thin,
                gray
            ]
            coordinates {(3+pi/4,0) (3+pi/4,5)};

            \addplot[
                dashed,
                thin,
                gray
            ]
            coordinates {(3+pi/2,0) (3+pi/2,5)};
            
            \node[below] at (axis cs:0,0)
            {$0$};
            
            \node[below] at (axis cs:pi/4,0)
            {$\tau$};
            
            \node[below] at (axis cs:3-pi/4,0)
            {$t-\tau$};
            
            \node[below] at (axis cs:3,0)
            {$t$};
            
            \node[below] at (axis cs:3+pi/4,0)
            {$t+\tau$};
            
            \node[below] at (axis cs:3+pi/2,0)
            {$t+2\tau$};
            
            \end{axis}
        \end{tikzpicture}
        \caption{Plot of the function $\chi$.}
        \label{fig:plot_chi}
    \end{figure}

    \begin{figure}[ht]
        \centering
        \begin{tikzpicture}[scale=0.8]
            \begin{axis}[
                width=12cm,
                height=7cm,
                domain=0:4.5708,
                samples=300,
                xmin=0,
                xmax=4.5708,
                ymin=-20,
                ymax=20,
                axis lines=none,
                xtick=\empty,
                ytick=\empty,
                xlabel={},
                ylabel={},
                clip=false,
            ]
            
            \draw[thick]
                (axis cs:0,0) -- (axis cs:3+pi/2,0);
            
            \addplot[
                domain=0:pi/4,
                smooth,
                thick,
                blue
            ]
            {cos(deg(2*x))};
            
            \addplot[
                domain=pi/4:3-pi/4,
                smooth,
                thick,
                purple
            ]
            {0};
            
            \addplot[
                domain=3-pi/4:3+pi/2,
                smooth,
                thick,
                green!50!black
            ]
            {exp(6-pi)*cos(deg(2*(3+pi/2-x)))};
            
            \addplot[
                dashed,
                thin,
                gray
            ]
            coordinates {(0,-20) (0,20)};
            
            \addplot[
                dashed,
                thin,
                gray
            ]
            coordinates {(pi/4,-20) (pi/4,20)};
            
            \addplot[
                dashed,
                thin,
                gray
            ]
            coordinates {(3-pi/4,-20) (3-pi/4,20)};
            
            \addplot[
                dashed,
                thin,
                gray
            ]
            coordinates {(3,-20) (3,20)};
            
            \addplot[
                dashed,
                thin,
                gray
            ]
            coordinates {(3+pi/4,-20) (3+pi/4,20)};

            \addplot[
                dashed,
                thin,
                gray
            ]
            coordinates {(3+pi/2,-20) (3+pi/2,20)};
            
            \node[below] at (axis cs:0,0)
            {$0$};
            
            \node[below] at (axis cs:pi/4,0)
            {$\tau$};
            
            \node[below] at (axis cs:3-pi/4,0)
            {$t-\tau$};
            
            \node[below] at (axis cs:3,0)
            {$t$};
            
            \node[below] at (axis cs:3+pi/4,0)
            {$t+\tau$};
            
            \node[below] at (axis cs:3+pi/2,0)
            {$t+2\tau$};
            
            \end{axis}
            \end{tikzpicture}
        \caption{Plot of the function $h = (\chi')^2-\Psi^2\chi^2$.}
        \label{fig:plot_h}
    \end{figure}
\end{proof}

\section*{Acknowledgements}

The authors were funded by the Deutsche Forschungsgemeinschaft (DFG, German Research Foundation) under Germany’s Excellence Strategy EXC 2047 - 390685813 as well as under CRC 1720 - 539309657. We thank Timo Lörke, Piro Man\c{c}o, Arnaud Guillin and Manon Michel for many fruitful discussions around the topics discussed here.
Moreover, we thank an anonymous reader of the first version of this paper for 
communicating to us the proof of the upper bound for the relaxation time in terms of the inverse singular value gap of the two-point motion based on Wei's Theorem, 
thus strengthening our original result which required regular initial laws. 
We believe that this proof was originally found with the help of AI but retain full responsibility for independently verifying and writing all statements and proofs.

\printbibliography

\end{document}